\documentclass[a4paper,oneside,10pt]{article}%
\usepackage{amsmath}
\usepackage{amsfonts}
\usepackage{amssymb}
\usepackage{graphicx}
\usepackage{color}
\usepackage[square,numbers,sort&compress]{natbib}%
\providecommand{\U}[1]{\protect\rule{.1in}{.1in}}

\newtheorem{theorem}{Theorem}[section]

\newtheorem{assumption}[theorem]{Assumption}

\newtheorem{lemma}[theorem]{Lemma}

\newenvironment{proof}[1][Proof]{\noindent\textbf{#1.} }{\ \rule{0.5em}{0.5em}}
\numberwithin{equation}{section}

\begin{document}

\title{A note on the global stochastic maximum principle for fully coupled
forward-backward stochastic systems}
\author{Mingshang Hu\thanks{Zhongtai Securities Institute for Financial Studies,
Shandong University, Jinan, Shandong 250100, PR China. humingshang@sdu.edu.cn.
Research supported by NSF (No. 11671231) and Young Scholars Program of
Shandong University (No. 2016WLJH10). }
\and Shaolin Ji\thanks{Zhongtai Securities Institute for Financial Studies,
Shandong University, Jinan, Shandong 250100, PR China. jsl@sdu.edu.cn
(Corresponding author). Research supported by NSF (No. 11571203).}
\and Xiaole Xue\thanks{Zhongtai Securities Institute for Financial Studies,
Shandong University, Jinan Shandong 250100, PR China. xiaolexue1989@gmail.com,
xuexiaole.good@163.com. Research supported by NSF (No. 11801315) and Natural
Science Foundation of Shandong Province(ZR2018QA001). }}
\maketitle

\textbf{Abstract}. Hu et. al \cite{Hu-JX} studied a stochastic optimal control
problem for fully coupled forward-backward stochastic control systems with a
nonempty control domain. By assuming a weakly coupled condition, they
established an approach to obtain the first-order, second-order variational
equations and the adjoint equations for the states $X$, $Y$ and $Z$ and
deduced the global maximum principle. But it is well known that there are
several different conditions such as monotonicity condition, weakly coupled
condition and other conditions (see \cite{YingHu002, Ma-WZZ, Ma-Yong-FBSDE,
Ma-ZZ,Pardoux-Tang,YongZhou, Zhang17} and the references therein) which can
guarantee the existence and uniqueness of the solution to (\ref{intro--fbsde}%
). In this note, to overcome the limitations of assuming a specific condition,
we propose two kinds of assumptions which can guarantee that the approach
developed in \cite{Hu-JX} is still applicable. Under these two kinds of
assumptions, we obtain the global stochastic maximum principle.



\addcontentsline{toc}{section}{\hspace*{1.8em}Abstract}

\section{Introduction}

In 1990, Peng \cite{Peng90} obtained the global maximum principle for the
classical stochastic optimal control problem. Since then, many researchers
investigate this kind of optimal control problems for various stochastic
systems (see \cite{YingHu006,YingHu001,Tang003,Tang004}). Peng \cite{Peng93}
generalized the classical stochastic optimal control problem to the so-called
stochastic recursive optimal control problem where the cost functional is
defined by $Y(0)$. Here $(Y(\cdot),Z(\cdot))$ is the solution of the following
backward stochastic differential equation (BSDE) (\ref{intro--bsde0}):%
\begin{equation}
\left\{
\begin{array}
[c]{rl}%
-dY(t)= & f(t,X(t),Y(t),Z(t),u(t))dt-Z(t)dB(t),\\
Y(T)= & \phi(X(T)).
\end{array}
\right.  \label{intro--bsde0}%
\end{equation}
In \cite{Peng93}, the control domain is convex and a local stochastic maximum
principle is established. The local stochastic maximum principles for other
various problems were studied in (Dokuchaev and Zhou \cite{Dokuchaev-Zhou}, Ji
and Zhou \cite{Ji-Zhou}, Peng \cite{Peng93}, Shi and Wu \cite{Shi-Wu}, Xu
\cite{Xu95}, Meyer-Brandis, {\O }ksendal and Zhou \cite{Zhou003}, see also the
references therein). When the control domain is nonconvex, one encounters an
essential difficulty when trying to derive the first-order and second-order
variational equations for the BSDE (\ref{intro--bsde0}) and it is proposed as
an open problem in Peng \cite{Peng99}. Hu \cite{Hu17} studied this open
problem and obtained a completely novel global maximum principle. Yong
\cite{Yong10} studied a fully coupled controlled FBSDE with mixed
initial-terminal conditions. In \cite{Yong10}, Yong regarded $Z(\cdot)$ as a
control process and then applied the Ekeland variational principle to obtain
an optimality variational principle which contains unknown parameters. Using
the similar approach, Wu \cite{Wu13} studied a stochastic recursive optimal
control problem.

In \cite{Hu-JX}, the following optimal control problem was considered:
minimize the cost functional
\[
J(u(\cdot))=Y(0)
\]
subject to the following fully coupled forward-backward stochastic
differential equation (FBSDE):
\begin{equation}
\left\{
\begin{array}
[c]{rl}%
dX(t)= & b(t,X(t),Y(t),Z(t),u(t))dt+\sigma(t,X(t),Y(t),Z(t),u(t))dB(t),\\
dY(t)= & -g(t,X(t),Y(t),Z(t),u(t))dt+Z(t)dB(t),\\
X(0)= & x_{0},\ Y(T)=\phi(X(T)),
\end{array}
\right.  \label{intro--fbsde}%
\end{equation}
where the control variable $u$ takes values in a nonempty subset of
$\mathbb{R}^{k}$ and the state variable $X$ belongs to $\mathbb{R}$. The
authors systematically developed an approach to obtain the first-order,
second-order variational equations and the adjoint equations for the states
$X$, $Y$ and $Z$ and deduced the global maximum principle.

To guarantee the well-posedness of (\ref{intro--fbsde}), a weakly coupled
condition was assumed in \cite{Hu-JX}. But it is well known that there are
several different conditions such as monotonicity condition, weakly coupled
condition and other conditions (see \cite{YingHu002, Ma-WZZ, Ma-Yong-FBSDE,
Ma-ZZ,Pardoux-Tang,YongZhou, Zhang17} and the references therein) which can
guarantee the existence and uniqueness of the solution to (\ref{intro--fbsde}%
). Then it naturally leads to the following problem: is the approach
established in \cite{Hu-JX} applicable to the other conditions except the
weakly coupled condition? After careful analysis, we found that applying the
approach in \cite{Hu-JX} to obtain the global maximum principle essentially
depends on the following assumptions: (1) there exists a unique solution to
FBSDE (\ref{intro--fbsde}); (2) the solution to FBSDE (\ref{intro--fbsde}) has
$L^{p}$-estimates; (3) there exists a unique solution to the first-order
adjoint equation. In other words, any assumptions which make the above three
statements hold are sufficient to deduce the global maximum principle by the
approach in \cite{Hu-JX}.

In this paper, motivated by the above analysis, we give up assuming a specific
condition (weakly coupled condition, monotonicity condition or other
conditions in the related literatures) and directly propose the following two
kind of assumptions. The first kind of assumptions is: (1) there exists a
unique solution to FBSDE (\ref{intro--fbsde}); (2) there exists a unique
bounded solution to the first-order adjoint equation. For this case, we can
prove the $L^{p}$-estimates for the solution to FBSDE (\ref{intro--fbsde})
hold. If the solution $q$ in the first-order adjoint equation is unbounded,
then the optimal control problem becomes more complicated. So for this case,
we propose the second kind of assumptions: (1) $\sigma$ is linear in $z$; (2)
there exists a unique solution to FBSDE (\ref{intro--fbsde}); (3) the solution
to FBSDE (\ref{intro--fbsde}) has $L^{p}$-estimates; (4) there exists a unique
solution to the first-order adjoint equation. We prove that for both kinds of
the assumptions, all the appropriate estimates for the solutions of the
first-order and second-order variational equations hold. Thus, the global
maximum principle can be deduced naturally. Beside this, we also generalize
the state variable $X$ in (\ref{intro--fbsde}) to multi-dimensional case in
this paper.

The rest of the paper is organized as follows. In section 2, we give the
preliminaries and formulation of our problem. A global stochastic maximum
principle is obtained by spike variation method in section 3. In appendix, we
give some results that will be used in our proofs.

\section{ Preliminaries and problem formulation}

Let $(\Omega,\mathcal{F},P)$ be a complete probability space on which a
standard $d$-dimensional Brownian motion $B=(B_{1}(t),B_{2}(t),...B_{d}%
(t))_{0\leq t\leq T}^{\intercal}$ is defined. Assume that $\mathbb{F=}%
\{\mathcal{F}_{t},0\leq t\leq T\}$ is the $P$-augmentation of the natural
filtration of $B$, where $\mathcal{F}_{0}$ contains all $P$-null sets of
$\mathcal{F}$. Denote by $\mathbb{R}^{n}$ the $n$-dimensional real Euclidean
space and $\mathbb{R}^{k\times n}$ the set of $k\times n$ real matrices. Let
$\langle\cdot,\cdot\rangle$ (resp. $\left\vert \cdot\right\vert $) denote the
usual scalar product (resp. usual norm) of $\mathbb{R}^{n}$ and $\mathbb{R}%
^{k\times n}$. The scalar product (resp. norm) of $M=(m_{ij})$, $N=(n_{ij}%
)\in\mathbb{R}^{k\times n}$ is denoted by $\langle M,N\rangle=\mathrm{tr}%
\{MN^{\intercal}\}$ (resp.$\Vert M\Vert=\sqrt{MM^{\intercal}}$), where the
superscript $^{\intercal}$ denotes the transpose of vectors or matrices.

We introduce the following spaces.

$L_{\mathcal{F}_{T}}^{p}(\Omega;\mathbb{R}^{n})$ : the space of $\mathcal{F}%
_{T}$-measurable $\mathbb{R}^{n}$-valued random variables $\eta$ such that
\[
||\eta||_{p}:=(\mathbb{E}[|\eta|^{p}])^{\frac{1}{p}}<\infty,
\]

$L_{\mathcal{F}_{T}}^{\infty}(\Omega;\mathbb{R}^{n})$: the space of
$\mathcal{F}_{T}$-measurable $\mathbb{R}^{n}$-valued random variables $\eta$
such that
\[
||\eta||_{\infty}:=\underset{\omega\in\Omega}{\mathrm{ess~sup}}\left\Vert
\eta\right\Vert <\infty,
\]

$L_{\mathcal{F}}^{p}([0,T];\mathbb{R}^{n})$: the space of $\mathbb{F}$-adapted
and $p$-th integrable stochastic processes on $[0,T]$ such that
\[
\mathbb{E}\left[  \int_{0}^{T}\left\vert f(t)\right\vert ^{p}dt\right]
<\infty,
\]

$L_{\mathcal{F}}^{\infty}(0,T;\mathbb{R}^{n})$: the space of $\mathbb{F}%
$-adapted and uniformly bounded stochastic processes on $[0,T]$ such that
\[
||f(\cdot)||_{\infty}=\underset{(t,\omega)\in\lbrack0,T]\times\Omega
}{\mathrm{ess~sup}}|f(t)|<\infty,
\]

$L_{\mathcal{F}}^{p,q}([0,T];\mathbb{R}^{n})$: the space of $\mathbb{F}%
$-adapted stochastic processes on $[0,T]$ such that
\[
||f(\cdot)||_{p,q}=\left\{  \mathbb{E}\left[  \left(  \int_{0}^{T}%
|f(t)|^{p}dt\right)  ^{\frac{q}{p}}\right]  \right\}  ^{\frac{1}{q}}<\infty,
\]

$L_{\mathcal{F}}^{p}(\Omega;C([0,T],\mathbb{R}^{n}))$: the space of
$\mathbb{F}$-adapted continuous stochastic processes on $[0,T]$ such that
\[
\mathbb{E}\left[  \sup\limits_{0\leq t\leq T}\left\vert f(t)\right\vert
^{p}\right]  <\infty.
\]

\subsection{Problem formulation}

Consider the following fully coupled stochastic control system:
\begin{equation}
\left\{
\begin{array}
[c]{rl}%
dX(t)= & b(t,X(t),Y(t),Z(t),u(t))dt+\sigma(t,X(t),Y(t),Z(t),u(t))dB(t),\\
dY(t)= & -g(t,X(t),Y(t),Z(t),u(t))dt+Z(t)dB(t),\\
X(0)= & x_{0},\ Y(T)=\phi(X(T)),
\end{array}
\right.  \label{state-eq}%
\end{equation}
where%
\[
b:[0,T]\times\mathbb{R}^{n}\times\mathbb{R}\times\mathbb{R}\times
U\rightarrow\mathbb{R}^{n},
\]%
\[
\sigma:[0,T]\times\mathbb{R}^{n}\times\mathbb{R}\times\mathbb{R}\times
U\rightarrow\mathbb{R}^{n\times1},
\]%
\[
g:[0,T]\times\mathbb{R}^{n}\times\mathbb{R}\times\mathbb{R}\times
U\rightarrow\mathbb{R},
\]%
\[
\phi:\mathbb{R}^{n}\rightarrow\mathbb{R}.
\]

An admissible control $u(\cdot)$ is an $\mathbb{F}$-adapted process with
values in $U$ such that%
\[
\sup\limits_{0\leq t\leq T}\mathbb{E}[|u(t)|^{8}]<\infty,
\]
where the control domain $U$ is a nonempty subset of $\mathbb{R}^{k}$. Denote
the admissible control set by $\mathcal{U}[0,T]$.

Our optimal control problem is to minimize the cost functional
\[
J(u(\cdot))=Y(0)
\]
over $\mathcal{U}[0,T]$, that is%

\begin{equation}
\underset{u(\cdot)\in\mathcal{U}[0,T]}{\inf}J(u(\cdot)). \label{obje-eq}%
\end{equation}

\section{Stochastic maximum principle}

We derive maximum principle (necessary condition for optimality) for the
optimization problem (\ref{obje-eq}) in this section. For simplicity of
presentation, we only study the case $d=1$. In this section, the constant $C$
will change from line to line in our proof.

\begin{assumption}
\label{assmlip} For $\psi=b,$ $\sigma,$ $g$ and $\phi$, we suppose

(i) $\psi$, $\psi_{x}$, $\psi_{y}$, $\psi_{z}$ are continuous in $(x,y,z,u)$;
$\psi_{x}$, $\psi_{y}$, $\psi_{z}$ are bounded; there exists a constant $L>0$
such that%
\[%
\begin{array}
[c]{rl}%
|\psi(t,x,y,z,u)| & \leq L\left(  1+|x|+|y|+|z|+|u|\right)  ,\\
|\sigma(t,0,0,z,u)-\sigma(t,0,0,z,u^{\prime})| & \leq L(1+|u|+|u^{\prime}|).
\end{array}
\]

(ii) $\psi_{xx}$, $\psi_{xy}$, $\psi_{yy}$ , $\psi_{xz}$, $\psi_{yz}$,
$\psi_{zz}$ are continuous in $(x,y,z,u)$; $\psi_{xx}$, $\psi_{xy}$,
$\psi_{yy}$, $\psi_{xz}$, $\psi_{yz}$ ,$\psi_{zz}$ are bounded.
\end{assumption}

\begin{assumption}
\label{assm-ex} For any $u(\cdot)\in\mathcal{U}[0,T]$ and $\beta\in
\lbrack2,8]$, FBSDE (\ref{state-eq}) has a unique solution $(X(\cdot
),Y(\cdot),Z(\cdot))\in L_{\mathcal{F}}^{\beta}(\Omega;C([0,T],\mathbb{R}%
^{n}))\times L_{\mathcal{F}}^{\beta}(\Omega;C([0,T],\mathbb{R}))\times
L_{\mathcal{F}}^{2,\beta}([0,T];\mathbb{R})$.
\end{assumption}

Let $\bar{u}(\cdot)$ be optimal and $(\bar{X}(\cdot),\bar{Y}(\cdot),\bar
{Z}(\cdot))$ be the corresponding state processes of (\ref{state-eq}). Since
the control domain is not necessarily convex, we resort to spike variation
method. For any $u(\cdot)\in\mathcal{U}[0,T]$ and $0<\epsilon<T$, define
\[
u^{\epsilon}(t)=\left\{
\begin{array}
[c]{ll}%
\bar{u}(t), & \ t\in\lbrack0,T]\backslash E_{\epsilon},\\
u(t), & \ t\in E_{\epsilon},
\end{array}
\right.
\]
where $E_{\epsilon}\subset\lbrack0,T]$ is\ a measurable set with
$|E_{\epsilon}|=\epsilon$. Let $(X^{\epsilon}(\cdot),Y^{\epsilon}%
(\cdot),Z^{\epsilon}(\cdot))$ be the state processes of (\ref{state-eq})
associated with $u^{\epsilon}(\cdot)$.

For simplicity, for $\psi=b$, $\sigma$, $g$, $\phi$ and $\kappa=x$, $y$, $z$,
denote%
\[%
\begin{array}
[c]{rl}%
\psi(t)= & \psi(t,\bar{X}(t),\bar{Y}(t),\bar{Z}(t),\bar{u}(t)),\\
\psi_{\kappa}(t)= & \psi_{\kappa}(t,\bar{X}(t),\bar{Y}(t),\bar{Z}(t),\bar
{u}(t)),\\
\delta\psi(t)= & \psi(t,\bar{X}(t),\bar{Y}(t),\bar{Z}(t),u(t))-\psi(t),\\
\delta\psi_{\kappa}(t)= & \psi_{\kappa}(t,\bar{X}(t),\bar{Y}(t),\bar
{Z}(t),u(t))-\psi_{\kappa}(t),\\
\delta\psi(t,\Delta)= & \psi(t,\bar{X}(t),\bar{Y}(t),\bar{Z}(t)+\Delta
(t),u(t))-\psi(t),\\
\delta\psi_{\kappa}(t,\Delta)= & \psi_{\kappa}(t,\bar{X}(t),\bar{Y}(t),\bar
{Z}(t)+\Delta(t),u(t))-\psi_{\kappa}(t),
\end{array}
\]
where $\Delta(\cdot)$ is an $\mathbb{F}$--adapted process. Moreover, denote
the gradient of $\psi$ with respect to $x$, $y$, $z$ by $D\psi$, and
$D^{2}\psi$ the Hessian matrix of $\psi$ with respect to $x$, $y$, $z$,%
\begin{align*}
D\psi(t)  &  =D\psi(t,\bar{X}(t),\bar{Y}(t),\bar{Z}(t),\bar{u}(t)),\\
D^{2}\psi(t)  &  =D^{2}\psi(t,\bar{X}(t),\bar{Y}(t),\bar{Z}(t),\bar{u}(t)).
\end{align*}

Let
\[%
\begin{array}
[c]{rl}%
\xi^{1,\epsilon}(t) & :=X^{\epsilon}(t)-\bar{X}(t);\text{ }\eta^{1,\epsilon
}(t):=Y^{\epsilon}(t)-\bar{Y}(t);\\
\zeta^{1,\epsilon}(t) & :=Z^{\epsilon}(t)-\bar{Z}(t);\text{ }\Theta
(t):=(\bar{X}(t),\bar{Y}(t),\bar{Z}(t));\\
\Theta^{\epsilon}(t) & :=(X^{\epsilon}(t),Y^{\epsilon}(t),Z^{\epsilon}(t)).
\end{array}
\]
We have
\begin{equation}
\left\{
\begin{array}
[c]{rl}%
d\xi^{1,\epsilon}(t)= & \left[  \tilde{b}_{x}^{\epsilon}(t)\xi^{1,\epsilon
}(t)+\tilde{b}_{y}^{\epsilon}(t)\eta^{1,\epsilon}(t)+\tilde{b}_{z}^{\epsilon
}(t)\zeta^{1,\epsilon}(t)+\delta b(t)I_{E_{\epsilon}}(t)\right]  dt\\
& +\left[  \tilde{\sigma}_{x}^{\epsilon}(t)\xi^{1,\epsilon}(t)+\tilde{\sigma
}_{y}^{\epsilon}(t)\eta^{1,\epsilon}(t)+\tilde{\sigma}_{z}^{\epsilon}%
(t)\zeta^{1,\epsilon}(t)+\delta\sigma(t)I_{E_{\epsilon}}(t)\right]  dB(t),\\
\xi^{1,\epsilon}(0)= & 0,
\end{array}
\right.  \label{ep-bar-x}%
\end{equation}%
\begin{equation}
\left\{
\begin{array}
[c]{rl}%
d\eta^{1,\epsilon}(t)= & -\left[  \left\langle \tilde{g}_{x}^{\epsilon}%
(t),\xi^{1,\epsilon}(t)\right\rangle +\tilde{g}_{y}^{\epsilon}(t)\eta
^{1,\epsilon}(t)+\tilde{g}_{z}^{\epsilon}(t)\zeta^{1,\epsilon}(t)+\delta
g(t)I_{E_{\epsilon}}(t)\right]  dt+\zeta^{1,\epsilon}(t)dB(t),\\
\eta^{1,\epsilon}(T)= & \left\langle \tilde{\phi}_{x}^{\epsilon}%
(T),\xi^{1,\epsilon}(T)\right\rangle ,
\end{array}
\right.  \label{ep-bar-y}%
\end{equation}
where
\begin{equation}
\tilde{b}_{x}^{\epsilon}(t)=\int_{0}^{1}b_{x}(t,\Theta(t)+\theta
(\Theta^{\epsilon}(t)-\Theta(t)),u^{\epsilon}(t))d\theta\label{eq-b-ep}%
\end{equation}
and $\tilde{b}_{y}^{\epsilon}(t)$, $\tilde{b}_{z}^{\epsilon}(t)$,
$\tilde{\sigma}_{x}^{\epsilon}(t)$, $\tilde{\sigma}_{y}^{\epsilon}(t)$,
$\tilde{\sigma}_{z}^{\epsilon}(t)$, $\tilde{g}_{x}^{\epsilon}(t)$, $\tilde
{g}_{y}^{\epsilon}(t)$, $\tilde{g}_{z}^{\epsilon}(t)$ and $\tilde{\phi}%
_{x}^{\epsilon}(T)$ are defined similarly. Consider the following linear
FBSDE
\begin{equation}
\left\{
\begin{array}
[c]{rl}%
d\hat{X}(t)= & \left[  \tilde{b}_{x}^{\epsilon}(t)\hat{X}(t)+\tilde{b}%
_{y}^{\epsilon}(t)\hat{Y}(t)+\tilde{b}_{z}^{\epsilon}(t)\hat{Z}(t)+L_{1}%
(t)\right]  dt+\left[  \tilde{\sigma}_{x}^{\epsilon}(t,\Delta)\hat
{X}(t)\right. \\
& \left.  +\tilde{\sigma}_{y}^{\epsilon}(t,\Delta)\hat{Y}(t)+\tilde{\sigma
}_{z}^{\epsilon}(t,\Delta)\hat{Z}(t)+L_{2}(t)\right]  dB(t),\\
d\hat{Y}(t)= & -\left[  \left\langle \tilde{g}_{x}^{\epsilon}(t),\hat
{X}(t)\right\rangle +\tilde{g}_{y}^{\epsilon}(t)\hat{Y}(t)+\tilde{g}%
_{z}^{\epsilon}(t)\hat{Z}(t)+L_{3}(t)\right]  dt+\hat{Z}(t)dB(t),\\
\hat{X}(0)= & x_{0},\ \hat{Y}(T)=\left\langle \tilde{\phi}_{x}^{\epsilon
}(T),\hat{X}(T)\right\rangle +\varsigma,
\end{array}
\right.  \label{eq-lfb}%
\end{equation}
where $\tilde{b}_{x}^{\epsilon}(t)$, $\tilde{b}_{y}^{\epsilon}(t)$, $\tilde
{b}_{z}^{\epsilon}(t)$, $\tilde{g}_{x}^{\epsilon}(t)$, $\tilde{g}%
_{y}^{\epsilon}(t)$, $\tilde{g}_{z}^{\epsilon}(t)$, $\tilde{\phi}%
_{x}^{\epsilon}(T)$ are defined as (\ref{eq-b-ep}) and $\tilde{\sigma}%
_{x}^{\epsilon}(t,\Delta)=\int_{0}^{1}\sigma_{x}(t,\Theta(t,\Delta
I_{E_{\epsilon}}(t))+\theta(\Theta^{\epsilon}(t)-\Theta(t,\Delta
I_{E_{\epsilon}}(t))),u^{\epsilon}(t))d\theta$ for any given $\Delta(\cdot)$,
$\tilde{\sigma}_{y}^{\epsilon}(t,\Delta)$, $\tilde{\sigma}_{z}^{\epsilon
}(t,\Delta)$ are defined similar to $\tilde{\sigma}_{x}^{\epsilon}(t,\Delta)$,
$L_{1}(\cdot)\in L_{\mathcal{F}}^{1,\beta}([0,T];\mathbb{R}^{n})$,
$L_{2}(\cdot)\in L_{\mathcal{F}}^{2,\beta}([0,T];\mathbb{R}^{n})$,
$L_{3}(\cdot)\in L_{\mathcal{F}}^{1,\beta}([0,T];\mathbb{R})$, $\varsigma\in
L_{\mathcal{F}_{T}}^{\beta}(\Omega;\mathbb{R})$. We impose the following assumption.

\begin{assumption}
\label{assm-ex-ln}For any $0<\epsilon<T$, $u^{\epsilon}(\cdot)\in
\mathcal{U}[0,T]$ and $\beta\in\lbrack2,8]$, the FBSDE (\ref{eq-lfb}) has a
unique solution $(\hat{X}(\cdot),\hat{Y}(\cdot),\hat{Z}(\cdot))\in
L_{\mathcal{F}}^{\beta}(\Omega;C([0,T],\mathbb{R}^{n}))\times L_{\mathcal{F}%
}^{\beta}(\Omega;C([0,T],\mathbb{R}))\times L_{\mathcal{F}}^{2,\beta
}([0,T];\mathbb{R})$.
\end{assumption}

\begin{assumption}
\label{assm-p-q}For any control $u^{\varepsilon}(\cdot)$ the following BSDE:%
\begin{equation}
\left\{
\begin{array}
[c]{rl}%
dp^{\epsilon}(t)= & -\{\tilde{g}_{x}^{\epsilon}(t)+\tilde{g}_{y}^{\epsilon
}(t)p^{\epsilon}(t)+\tilde{g}_{z}^{\epsilon}(t)K_{1}^{\epsilon}(t)+\tilde
{b}_{x}^{\epsilon}(t)^{\intercal}p^{\epsilon}(t)+\left\langle p^{\epsilon
}(t),\tilde{b}_{y}^{\epsilon}(t)\right\rangle p^{\epsilon}(t)+\left\langle
p^{\epsilon}(t),\tilde{b}_{z}^{\epsilon}(t)\right\rangle K_{1}^{\epsilon}(t)\\
& +\tilde{\sigma}_{x}^{\epsilon}(t,\Delta)^{\intercal}q^{\epsilon
}(t)+\left\langle q^{\epsilon}(t),\tilde{\sigma}_{y}^{\epsilon}(t,\Delta
)\right\rangle p^{\epsilon}(t)+\left\langle q^{\epsilon}(t),\tilde{\sigma}%
_{z}^{\epsilon}(t,\Delta)\right\rangle K_{1}^{\epsilon}(t)\}dt+q^{\epsilon
}(t)dB(t)\\
p^{\epsilon}(T)= & \tilde{\phi}_{x}^{\epsilon}(T),
\end{array}
\right.  \label{eq-p-q-b}%
\end{equation}
where
\[
K_{1}^{\epsilon}(t)=\left(  1-\left\langle p(t),\sigma_{z}(t,\Delta
)\right\rangle \right)  ^{-1}\left[  \sigma_{x}(t)^{\intercal}p^{\epsilon
}(t)+\left\langle p^{\epsilon}(t),\sigma_{y}(t)\right\rangle p^{\epsilon
}(t)+q^{\epsilon}(t)\right]
\]
has a unique solution $(p^{\varepsilon}(\cdot),q^{\varepsilon}(\cdot))\in
L_{\mathcal{F}}^{\infty}(\Omega;C([0,T],\mathbb{R}^{n}))\times L_{\mathcal{F}%
}^{\infty}([0,T];\mathbb{R}^{n})$ such that $\left\vert 1-\left\langle
p^{\varepsilon}(t),\gamma_{2}(t)\right\rangle \right\vert ^{-1}$ is uniformly bounded.
\end{assumption}

Note that $\tilde{\sigma}_{x}^{\epsilon}(t,\Delta)=\tilde{\sigma}%
_{x}^{\epsilon}(t)$ when $\Delta(\cdot)\equiv0$. Due to Assumption
\ref{assm-ex-ln}, there exists a unique solution $(\xi^{1,\epsilon}(\cdot
)$,$\eta^{1,\epsilon}(\cdot)$,$\zeta^{1,\epsilon}(\cdot))$ to (\ref{ep-bar-x})
and (\ref{ep-bar-y}).

\begin{lemma}
\label{est-epsilon-bar}Suppose that Assumptions \ref{assmlip}, \ref{assm-ex}%
,\ref{assm-ex-ln} and \ref{assm-p-q} hold. Then for any $2\leq\beta\leq8$ we
have
\begin{equation}
\mathbb{E}\left[  \sup\limits_{t\in\lbrack0,T]}\left(  |X^{\epsilon}%
(t)-\bar{X}(t)|^{\beta}+|Y^{\epsilon}(t)-\bar{Y}(t)|^{\beta}\right)  \right]
+\mathbb{E}\left[  \left(  \int_{0}^{T}|Z^{\epsilon}(t)-\bar{Z}(t)|^{2}%
dt\right)  ^{\frac{\beta}{2}}\right]  =O\left(  \epsilon^{\frac{\beta}{2}%
}\right)  .
\end{equation}

\end{lemma}

\begin{proof}
Note that $\left(  \xi^{1,\epsilon}(t),\eta^{1,\epsilon}(t),\zeta^{1,\epsilon
}(t)\right)  $ is the solution to (\ref{ep-bar-x}) and (\ref{ep-bar-y}), and
\[
\mathbb{E}\left[  \left(  \int_{E_{\epsilon}}|u(t)|dt\right)  ^{\beta}\right]
\leq\epsilon^{\beta-1}\mathbb{E}\left[  \int_{E_{\epsilon}}|u(t)|^{\beta
}dt\right]  .
\]
Then, by Lemma \ref{est-l} in Appendix, we get
\[%
\begin{array}
[c]{l}%
\mathbb{E}\left[  \sup\limits_{t\in\lbrack0,T]}\left(  |\xi^{1,\epsilon
}(t)|^{\beta}+|\eta^{1,\epsilon}(t)|^{\beta}\right)  +\left(  \int_{0}%
^{T}|\zeta^{1,\epsilon}(t)|^{2}dt\right)  ^{\frac{\beta}{2}}\right] \\
\ \ \leq C\mathbb{E}\left[  \left(  \int_{0}^{T}\left(  |\delta
b(t)|I_{E_{\epsilon}}(t)+|\delta g(t)|I_{E_{\epsilon}}(t)\right)  dt\right)
^{\beta}+\left(  \int_{0}^{T}|\delta\sigma(t)|^{2}I_{E_{\epsilon}%
}(t)dt\right)  ^{\frac{\beta}{2}}\right] \\
\ \ \leq C\mathbb{E}\left[  \left(  \int_{E_{\epsilon}}(1+|\bar{X}%
(t)|+|\bar{Y}(t)|+|\bar{Z}(t)|+|u(t)|+|\bar{u}(t)|)dt\right)  ^{\beta}\right.
\\
\text{ \ \ \ \ \ \ \ }\left.  +\left(  \int_{E_{\epsilon}}(1+|\bar{X}%
(t)|^{2}+|\bar{Y}(t)|^{2}+|u(t)|^{2}+|\bar{u}(t)|^{2})dt\right)  ^{\frac
{\beta}{2}}\right] \\
\ \ \leq C\left(  \epsilon^{\beta}+\epsilon^{\frac{\beta}{2}}\right)  \left(
1+\sup\limits_{t\in\lbrack0,T]}\mathbb{E}\left[  |\bar{X}(t)|^{\beta}+|\bar
{Y}(t)|^{\beta}+|u(t)|^{\beta}+|\bar{u}(t)|^{\beta}\right]  \right)
+C\epsilon^{\frac{\beta}{2}}\mathbb{E}\left[  \left(  \int_{0}^{T}|\bar
{Z}(t)|^{2}dt\right)  ^{\frac{\beta}{2}}\right] \\
\ \ \leq C\epsilon^{\frac{\beta}{2}}.
\end{array}
\]

\end{proof}

\subsection{First-order expansion}

We introduce the following adjoint equation satisfied by $\left(  p,q\right)
$:%
\begin{equation}
\left\{
\begin{array}
[c]{l}%
dp(t)\\
=-\left\{  g_{x}(t)+g_{y}(t)p(t)+g_{z}(t)K_{1}(t)+b_{x}(t)^{\intercal
}p(t)+\left\langle p(t),b_{y}(t)\right\rangle p(t)\right. \\
\text{ \ }\left.  +\left\langle p(t),b_{z}(t)\right\rangle K_{1}(t)+\sigma
_{x}(t)^{\intercal}q(t)+\left\langle q(t),\sigma_{y}(t)\right\rangle
p(t)+\left\langle q(t),\sigma_{z}(t)\right\rangle K_{1}(t)\right\}
dt+q(t)dB(t),\\
p(T)=\phi_{x}(\bar{X}(T)),
\end{array}
\right.  \label{eq-p}%
\end{equation}
where%
\begin{equation}
K_{1}(t)=\left(  1-\left\langle p(t),\sigma_{z}(t)\right\rangle \right)
^{-1}\left[  \sigma_{x}(t)^{\intercal}p(t)+\left\langle p(t),\sigma
_{y}(t)\right\rangle p(t)+q(t)\right]  . \label{def-k1}%
\end{equation}
We first study the following algebra equation
\begin{equation}
\Delta(t)=p(t)(\sigma(t,\bar{X}(t),\bar{Y}(t),\bar{Z}(t)+\Delta
(t),u(t))-\sigma(t,\bar{X}(t),\bar{Y}(t),\bar{Z}(t),\bar{u}(t))),\;t\in
\lbrack0,T], \label{def-delt}%
\end{equation}
where $u\left(  \cdot\right)  $ is a given admissible control.

\begin{assumption}
\label{assm-delta}Assume that equation (\ref{def-delt}) has a unique solution
$\Delta\left(  \cdot\right)  $, and it satisfies%
\begin{equation}
|\Delta(t)|\leq C(1+|\bar{X}(t)|+|\bar{Y}(t)|+|u(t)|+|\bar{u}(t)|),\text{
}t\in\lbrack0,T], \label{del-new-3}%
\end{equation}
where $C$ is a constant depending on $\beta_{0}$, $L$, $\left\Vert \psi
_{x}\right\Vert _{\infty}$,$\left\Vert \psi_{y}\right\Vert _{\infty}$,
$\left\Vert \psi_{z}\right\Vert _{\infty}$, $T$.
\end{assumption}

Now we introduce the first-order variational equation:
\begin{equation}
\left\{
\begin{array}
[c]{rl}%
dX_{1}(t)= & \left[  b_{x}(t)X_{1}(t)+b_{y}(t)Y_{1}(t)+b_{z}(t)(Z_{1}%
(t)-\Delta(t)I_{E_{\epsilon}}(t))\right]  dt\\
& +\left[  \sigma_{x}(t)X_{1}(t)+\sigma_{y}(t)Y_{1}(t)+\sigma_{z}%
(t)(Z_{1}(t)-\Delta(t)I_{E_{\epsilon}}(t))+\delta\sigma(t,\Delta
)I_{E_{\epsilon}}(t)\right]  dB(t),\\
X_{1}(0)= & 0,
\end{array}
\right.  \label{new-form-x1}%
\end{equation}
and
\begin{equation}
\left\{
\begin{array}
[c]{l}%
dY_{1}(t)=-\left[  \left\langle g_{x}(t),X_{1}(t)\right\rangle +g_{y}%
(t)Y_{1}(t)+g_{z}(t)(Z_{1}(t)-\Delta(t)I_{E_{\epsilon}}(t))-\left\langle
q(t),\delta\sigma(t,\Delta)\right\rangle I_{E_{\epsilon}}(t)\right]  dt\\
\text{ \ \ \ \ \ \ \ \ \ \ }+Z_{1}(t)dB(t),\\
Y_{1}(T)=\left\langle \phi_{x}(\bar{X}(T)),X_{1}(T)\right\rangle .
\end{array}
\right.  \label{new-form-y1}%
\end{equation}

By Assumption \ref{assm-ex-ln}, the above FBSDE has a unique solution
$(X_{1}(\cdot),Y_{1}(\cdot),Z_{1}(\cdot))$.

\begin{lemma}
\label{lemma-y1}Suppose that Assumptions \ref{assmlip}, \ref{assm-ex},
\ref{assm-ex-ln}, \ref{assm-p-q} and \ref{assm-delta} hold. Then we have
\begin{align*}
Y_{1}(t)  &  =\left\langle p(t),X_{1}(t)\right\rangle ,\\
Z_{1}(t)  &  =\left\langle K_{1}(t),X_{1}(t)\right\rangle +\Delta
(t)I_{E_{\epsilon}}(t),
\end{align*}
where $p(\cdot)$ is the solution of (\ref{eq-p}) and $K_{1}\left(
\cdot\right)  $ is given in (\ref{def-k1}).
\end{lemma}

\begin{proof}
From Lemma \ref{appen-th-linear-fbsde} in Appendix, we can obtain the desired results.
\end{proof}

Let
\[%
\begin{array}
[c]{rl}%
\xi^{2,\epsilon}(t) & :=X^{\epsilon}(t)-\bar{X}(t)-X_{1}(t);\text{ }%
\eta^{2,\epsilon}(t):=Y^{\epsilon}(t)-\bar{Y}(t)-Y_{1}(t);\\
\zeta^{2,\epsilon}(t) & :=Z^{\epsilon}(t)-\bar{Z}(t)-Z_{1}(t);\text{ }%
\Theta(t):=(\bar{X}(t),\bar{Y}(t),\bar{Z}(t)).
\end{array}
\]
Then we have the following estimates.

\begin{lemma}
\label{est-one-order}Suppose Assumptions \ref{assmlip}, \ref{assm-ex},
\ref{assm-ex-ln}, \ref{assm-p-q} and \ref{assm-delta} hold. Then for any
$2\leq\beta\leq8$, we have the following estimates
\begin{equation}
\mathbb{E}\left[  \sup\limits_{t\in\lbrack0,T]}\left(  |X_{1}(t)|^{\beta
}+|Y_{1}(t)|^{\beta}\right)  \right]  +\mathbb{E}\left[  \left(  \int_{0}%
^{T}|Z_{1}(t)|^{2}dt\right)  ^{\beta/2}\right]  =O(\epsilon^{\beta/2}),
\label{est-x1-y1}%
\end{equation}%
\[
\mathbb{E}\left[  \sup\limits_{t\in\lbrack0,T]}\left(  |X^{\epsilon}%
(t)-\bar{X}(t)-X_{1}(t)|^{2}+|Y^{\epsilon}(t)-\bar{Y}(t)-Y_{1}(t)|^{2}\right)
\right]  +\mathbb{E}\left[  \int_{0}^{T}|Z^{\epsilon}(t)-\bar{Z}%
(t)-Z_{1}(t)|^{2}dt\right]  =O(\epsilon^{2}),
\]%
\[
\mathbb{E}\left[  \sup\limits_{t\in\lbrack0,T]}(|X^{\epsilon}(t)-\bar
{X}(t)-X_{1}(t)|^{4}+|Y^{\epsilon}(t)-\bar{Y}(t)-Y_{1}(t)|^{4})\right]
+\mathbb{E}\left[  \left(  \int_{0}^{T}|Z^{\epsilon}(t)-\bar{Z}(t)-Z_{1}%
(t)|^{2}dt\right)  ^{2}\right]  =o(\epsilon^{2}).
\]

\end{lemma}

\begin{proof}
{By Lemma \ref{est-l} in Appendix,} we have
\[%
\begin{array}
[c]{rl}
& \mathbb{E}\left[  \sup\limits_{t\in\lbrack0,T]}\left(  |X_{1}(t)|^{\beta
}+|Y_{1}(t)|^{\beta}\right)  +\left(  \int_{0}^{T}|Z_{1}(t)|^{2}dt\right)
^{\beta/2}\right] \\
& \leq C\mathbb{E}\left[  \left(  \int_{0}^{T}\left\vert \delta\sigma
(t,\Delta)+\Delta(t)\right\vert ^{2}I_{E_{\epsilon}}(t)dt\right)
^{\frac{\beta}{2}}\right] \\
& \leq C\mathbb{E}\left[  \left(  \int_{E_{\epsilon}}\left(  1+|\bar
{X}(t)|^{2}+\left\vert \bar{Y}(t)\right\vert ^{2}+\left\vert \bar{u}%
(t)|^{2}+|u(t)\right\vert ^{2}\right)  dt\right)  ^{\beta/2}\right] \\
& \leq C\epsilon^{^{\beta/2}}.
\end{array}
\]
We use the notations $\xi^{1,\epsilon}(t)$, $\eta^{1,\epsilon}(t)$ and
$\zeta^{1,\epsilon}(t)$ in the proof of Lemma \ref{est-epsilon-bar} and
\[%
\begin{array}
[c]{rl}%
\Theta(t,\Delta I_{E_{\epsilon}}) & :=(\bar{X}(t),\bar{Y}(t),\bar{Z}%
(t)+\Delta(t)I_{E_{\epsilon}}(t));\text{ }\Theta^{\epsilon}(t):=(X^{\epsilon
}(t),Y^{\epsilon}(t),Z^{\epsilon}(t)).
\end{array}
\]
Note that
\[%
\begin{array}
[c]{l}%
\delta\sigma(t,\Delta)I_{E_{\epsilon}}(t)=\sigma(t,\bar{X}(t),\bar{Y}%
(t),\bar{Z}(t)+\Delta(t)I_{E_{\epsilon}}(t),u^{\epsilon}(t))-\sigma
(t)=\sigma(t,\Theta(t,\Delta I_{E_{\epsilon}}(t)),u^{\epsilon}(t))-\sigma(t).
\end{array}
\]
We have%
\[%
\begin{array}
[c]{l}%
\sigma(t,\Theta^{\epsilon}(t),u^{\epsilon}(t))-\sigma(t)-\delta\sigma
(t,\Delta)I_{E_{\epsilon}}(t)\\
=\sigma(t,\Theta^{\epsilon}(t),u^{\epsilon}(t))-\sigma(t,\Theta(t,\Delta
I_{E_{\epsilon}}(t)),u^{\epsilon}(t))\\
=\tilde{\sigma}_{x}^{\epsilon}(t)\left(  X^{\epsilon}(t)-\bar{X}(t)\right)
+\tilde{\sigma}_{y}^{\epsilon}(t)(Y^{\epsilon}(t)-\bar{Y}(t))+\tilde{\sigma
}_{z}^{\epsilon}(t)(Z^{\epsilon}(t)-\bar{Z}(t)-\Delta(t)I_{E_{\epsilon}}(t)),
\end{array}
\]
where
\[
\tilde{\sigma}_{x}^{\epsilon}(t,\Delta)=\int_{0}^{1}\sigma_{x}(t,\Theta
(t,\Delta I_{E_{\epsilon}}(t))+\theta(\Theta^{\epsilon}(t)-\Theta(t,\Delta
I_{E_{\epsilon}}(t))),u^{\epsilon}(t))d\theta,
\]
and $\tilde{\sigma}_{y}^{\epsilon}(t,\Delta)$, $\tilde{\sigma}_{z}^{\epsilon
}(t,\Delta)$\ are defined similarly. \ 

Recall that $\tilde{b}_{x}^{\epsilon}(t)$, $\tilde{b}_{y}^{\epsilon}(t)$,
$\tilde{b}_{z}^{\epsilon}(t)$, $\tilde{g}_{x}^{\epsilon}(t)$, $\tilde{g}%
_{y}^{\epsilon}(t)$, $\tilde{g}_{z}^{\epsilon}(t)$ and $\tilde{\phi}%
_{x}^{\epsilon}(T)$ are defined in Lemma \ref{est-epsilon-bar}. Then,%
\begin{equation}
\left\{
\begin{array}
[c]{ll}%
d\xi^{2,\epsilon}(t)= & \left[  \tilde{b}_{x}^{\epsilon}(t)\xi^{2,\epsilon
}(t)+\tilde{b}_{y}^{\epsilon}(t)\eta^{2,\epsilon}(t)+\tilde{b}_{z}^{\epsilon
}(t)\zeta^{2,\epsilon}(t)+A_{1}^{\epsilon}(t)\right]  dt\\
& +\left[  \tilde{\sigma}_{x}^{\epsilon}(t,\Delta)\xi^{2,\epsilon}%
(t)+\tilde{\sigma}_{y}^{\epsilon}(t,\Delta)\eta^{2,\epsilon}(t)+\tilde{\sigma
}_{z}^{\epsilon}(t,\Delta)\zeta^{2,\epsilon}(t))+B_{1}^{\epsilon}(t)\right]
dB(t),\\
\xi^{2,\epsilon}(0)= & 0,
\end{array}
\right.  \label{deri-x-x1}%
\end{equation}%
\[
\left\{
\begin{array}
[c]{rl}%
d\eta^{2,\epsilon}(t)= & -\left[  \left\langle \tilde{g}_{x}^{\epsilon}%
(t),\xi^{2,\epsilon}(t)\right\rangle +\tilde{g}_{y}^{\epsilon}(t)\eta
^{2,\epsilon}(t)+\tilde{g}_{z}^{\epsilon}(t)\zeta^{2,\epsilon}(t)+C_{1}%
^{\epsilon}(t)\right]  dt+\zeta^{2,\epsilon}(t)dB(t),\\
\eta^{2,\epsilon}(T)= & \left\langle \tilde{\phi}_{x}^{\epsilon}%
(T),\xi^{2,\epsilon}(T)\right\rangle +D_{1}^{\epsilon}(T),
\end{array}
\right.
\]
{where%
\[%
\begin{array}
[c]{rl}%
A_{1}^{\epsilon}(t)= & (\tilde{b}_{x}^{\epsilon}(t)-b_{x}(t))X_{1}%
(t)+(\tilde{b}_{y}^{\epsilon}(t)-b_{y}(t))Y_{1}(t)+(\tilde{b}_{z}^{\epsilon
}(t)-b_{z}(t))Z_{1}(t)\\
& +b_{z}(t)\Delta(t)I_{E_{\epsilon}}(t)+\delta b(t)I_{E_{\epsilon}}(t),\\
B_{1}^{\epsilon}(t)= & (\tilde{\sigma}_{x}^{\epsilon}(t,\Delta)-\sigma
_{x}(t))X_{1}(t)+(\tilde{\sigma}_{y}^{\epsilon}(t,\Delta)-\sigma_{y}%
(t))Y_{1}(t)+(\tilde{\sigma}_{z}^{\epsilon}(t,\Delta)-\sigma_{z}%
(t))\left\langle K_{1}(t),X_{1}(t)\right\rangle ,\\
C_{1}^{\epsilon}(t)= & \left\langle (\tilde{g}_{x}^{\epsilon}(t)-g_{x}%
(t),X_{1}(t)\right\rangle +(\tilde{g}_{y}^{\epsilon}(t)-g_{y}(t))Y_{1}%
(t)+(\tilde{g}_{z}^{\epsilon}(t)-g_{z}(t))Z_{1}(t)+\delta g(t)I_{E_{\epsilon}%
}(t)\\
& +g_{z}(t)\Delta(t)I_{E_{\epsilon}}(t)+\left\langle q(t),\delta
\sigma(t,\Delta)\right\rangle I_{E_{\epsilon}}(t),\\
D_{1}^{\epsilon}(T)= & \left\langle \tilde{\phi}_{x}^{\epsilon}(T)-\phi
_{x}(\bar{X}(T)),X_{1}(T)\right\rangle .
\end{array}
\]
By Lemma \ref{est-l} in Appendix, we obtain
\[%
\begin{array}
[c]{l}%
\mathbb{E}\left[  \sup\limits_{t\in\lbrack0,T]}\left(  |\xi^{2,\epsilon
}(t)|^{2}+|\eta^{2,\epsilon}(t)|^{2}\right)  +\int_{0}^{T}|\zeta^{2,\epsilon
}(t)|^{2}dt\right] \\
\leq C\mathbb{E}\left[  \left(  \int_{0}^{T}\left(  |A_{1}^{\epsilon
}(t)|+|C_{1}^{\epsilon}(t)|\right)  dt\right)  ^{2}+\int_{0}^{T}%
|B_{1}^{\epsilon}(t)|^{2}dt+|D_{1}^{\epsilon}(T)|^{2}\right] \\
\leq C\mathbb{E}\left[  \left(  \int_{0}^{T}|A_{1}^{\epsilon}(t)|dt\right)
^{2}+\left(  \int_{0}^{T}|C_{1}^{\epsilon}(t)|dt\right)  ^{2}+\int_{0}%
^{T}|B_{1}^{\epsilon}(t)|^{2}dt+|D_{1}^{\epsilon}(T)|^{2}\right]  .
\end{array}
\]
The following proof of the estimates are the same as in \cite{Hu-JX}. }
\end{proof}

\subsection{Second-order expansion}

Noting that $Z_{1}(t)=K_{1}(t)X_{1}(t)+\Delta(t)I_{E_{\epsilon}}(t)$\ in Lemma
\ref{lemma-y1}, then we introduce the second-order variational equation as
follows:
\begin{equation}
\left\{
\begin{array}
[c]{rl}%
dX_{2}(t)= & \left\{  b_{x}(t)X_{2}(t)+b_{y}(t)Y_{2}(t)+b_{z}(t)Z_{2}%
(t)+\delta b(t,\Delta)I_{E_{\epsilon}}(t)\right. \\
& \left.  +\frac{1}{2}D^{2}b(t)\left(  X_{1}(t)^{\intercal},Y_{1}%
(t),\left\langle K_{1}(t),X_{1}(t)\right\rangle \right)  ^{2}\right\}  dt\\
& +\left\{  \sigma_{x}(t)X_{2}(t)+\sigma_{y}(t)Y_{2}(t)+\frac{1}{2}D^{2}%
\sigma(t)\left(  X_{1}(t)^{\intercal},Y_{1}(t),\left\langle K_{1}%
(t),X_{1}(t)\right\rangle \right)  ^{2}\right. \\
& \left.  +\sigma_{z}(t)Z_{2}(t)+\left[  \delta\sigma_{x}(t,\Delta
)X_{1}(t)+\delta\sigma_{y}(t,\Delta)Y_{1}(t)\right]  I_{E_{\epsilon}%
}(t)+\delta\sigma_{z}(t,\Delta)\left\langle K_{1}(t),X_{1}(t)\right\rangle
I_{E_{\epsilon}}(t)\right\}  dB(t),\\
X_{2}(0)= & 0,
\end{array}
\right.  \label{new-form-x2}%
\end{equation}
and
\begin{equation}
\left\{
\begin{array}
[c]{ll}%
dY_{2}(t)= & -\left\{  \left\langle g_{x}(t),X_{2}(t)\right\rangle
+g_{y}(t)Y_{2}(t)+g_{z}(t)Z_{2}(t)+\left[  \left\langle q(t),\delta
\sigma(t,\Delta)\right\rangle +\delta g(t,\Delta)\right]  I_{E_{\epsilon}%
}(t)\right. \\
& \left.  +\frac{1}{2}\left[  X_{1}(t)^{\intercal},Y_{1}(t),\left\langle
K_{1}(t),X_{1}(t)\right\rangle \right]  D^{2}g(t)\left[  X_{1}(t)^{\intercal
},Y_{1}(t),\left\langle K_{1}(t),X_{1}(t)\right\rangle \right]  ^{\intercal
}\right\}  dt+Z_{2}(t)dB(t),\\
Y_{2}(T)= & \left\langle \phi_{x}(\bar{X}(T)),X_{2}(T)\right\rangle +\frac
{1}{2}\left\langle \phi_{xx}(\bar{X}(T))X_{1}(T),X_{1}(T)\right\rangle ,
\end{array}
\right.  \label{new-form-y2}%
\end{equation}
where
\[%
\begin{array}
[c]{l}%
D^{2}b(t)\left(  X_{1}(t)^{\intercal},Y_{1}(t),\left\langle K_{1}%
(t),X_{1}(t)\right\rangle \right)  ^{2}\\
=(\mathrm{tr}[D^{2}b^{1}(t)\left(  X_{1}(t)^{\intercal},Y_{1}(t),\left\langle
K_{1}(t),X_{1}(t)\right\rangle \right)  \left(  X_{1}(t)^{\intercal}%
,Y_{1}(t),\left\langle K_{1}(t),X_{1}(t)\right\rangle \right)  ^{\intercal
}],...,\\
\text{ \ \ \ \ }\mathrm{tr}[D^{2}b^{n}(t)\left(  X_{1}(t)^{\intercal}%
,Y_{1}(t),\left\langle K_{1}(t),X_{1}(t)\right\rangle \right)  \left(
X_{1}(t)^{\intercal},Y_{1}(t),\left\langle K_{1}(t),X_{1}(t)\right\rangle
\right)  ^{\intercal}]^{\intercal})
\end{array}
\]
and $D^{2}\sigma(t)\left(  X_{1}(t)^{\intercal},Y_{1}(t),\left\langle
K_{1}(t),X_{1}(t)\right\rangle \right)  ^{2}$ is defined similarly. In the
following lemma, we estimate the orders of $X_{2}(\cdot)$, $Y_{2}\left(
\cdot\right)  $, $Z_{2}\left(  \cdot\right)  $, and $Y^{\epsilon}(0)-\bar
{Y}(0)-Y_{1}(0)-Y_{2}(0)$. Let
\[%
\begin{array}
[c]{rl}%
\xi^{3,\epsilon}(t) & :=X^{\epsilon}(t)-\bar{X}(t)-X_{1}(t)-X_{2}(t);\text{
}\eta^{3,\epsilon}(t):=Y^{\epsilon}(t)-\bar{Y}(t)-Y_{1}(t)-Y_{2}(t);\\
\zeta^{3,\epsilon}(t) & :=Z^{\epsilon}(t)-\bar{Z}(t)-Z_{1}(t)-Z_{2}(t);\text{
}\Theta(t):=(\bar{X}(t),\bar{Y}(t),\bar{Z}(t)).
\end{array}
\]

\begin{lemma}
\label{est-second-order} Suppose that Assumption \ref{assmlip}, \ref{assm-ex},
\ref{assm-ex-ln}, \ref{assm-p-q} and \ref{assm-delta} hold. Then for any
$2\leq\beta\leq4$ we have%
\[%
\begin{array}
[c]{rl}%
\mathbb{E}\left[  \sup\limits_{t\in\lbrack0,T]}(|X_{2}(t)|^{2}+|Y_{2}%
(t)|^{2})\right]  +\mathbb{E}\left[  \int_{0}^{T}|Z_{2}(t)|^{2}dt\right]  &
=O(\epsilon^{2}),\\
\mathbb{E}\left[  \sup\limits_{t\in\lbrack0,T]}(|X_{2}(t)|^{\beta}%
+|Y_{2}(t)|^{\beta})\right]  +\mathbb{E}\left[  \left(  \int_{0}^{T}%
|Z_{2}(t)|^{2}dt\right)  ^{\frac{\beta}{2}}\right]  & =o(\epsilon^{\frac
{\beta}{2}}),\\
Y^{\epsilon}(0)-\bar{Y}(0)-Y_{1}(0)-Y_{2}(0) & =o(\epsilon).
\end{array}
\]

\end{lemma}

\begin{proof}
Let
\[%
\begin{array}
[c]{rl}%
L_{1}(t)= & \delta b(t,\Delta)I_{E_{\epsilon}}(t)+\frac{1}{2}D^{2}b(t)\left(
X_{1}(t)^{\intercal},Y_{1}(t),\left\langle K_{1}(t),X_{1}(t)\right\rangle
\right)  ^{2},\\
L_{2}(t)= & \frac{1}{2}D^{2}\sigma(t)\left(  X_{1}(t)^{\intercal}%
,Y_{1}(t),\left\langle K_{1}(t),X_{1}(t)\right\rangle \right)  ^{2}+\left[
\delta\sigma_{x}(t,\Delta)X_{1}(t)+\delta\sigma_{y}(t,\Delta)Y_{1}(t)\right]
I_{E_{\epsilon}}(t)\\
& +\delta\sigma_{z}(t,\Delta)\left\langle K_{1}(t),X_{1}(t)\right\rangle
I_{E_{\epsilon}}(t),\\
L_{3}\left(  t\right)  = & \left[  \left\langle q(t),\delta\sigma
(t,\Delta)\right\rangle +\delta g(t,\Delta)\right]  I_{E_{\epsilon}}%
(t)+\frac{1}{2}\left[  X_{1}(t)^{\intercal},Y_{1}(t),\left\langle
K_{1}(t),X_{1}(t)\right\rangle \right]  D^{2}g(t)\left[  X_{1}(t)^{\intercal
},Y_{1}(t),\left\langle K_{1}(t),X_{1}(t)\right\rangle \right]  ^{\intercal
},\\
\varsigma= & \frac{1}{2}\left\langle \phi_{xx}(\bar{X}(T))X_{1}(T),X_{1}%
(T)\right\rangle .
\end{array}
\]
By Lemma \ref{est-l} in Appendix, we have
\[%
\begin{array}
[c]{l}%
\mathbb{E}\left[  \sup\limits_{t\in\lbrack0,T]}(|X_{2}(t)|^{2}+|Y_{2}%
(t)|^{2})\right]  +\mathbb{E}\left[  \int_{0}^{T}|Z_{2}(t)|^{2}dt\right] \\
\leq C\mathbb{E}\left[  \left\vert \varsigma\right\vert ^{2}+\left(  \int%
_{0}^{T}\left(  \left\vert L_{1}\left(  t\right)  \right\vert +\left\vert
L_{3}\left(  t\right)  \right\vert \right)  dt\right)  ^{2}+\int_{0}%
^{T}\left\vert L_{2}\left(  t\right)  \right\vert ^{2}dt\right] \\
\leq C\mathbb{E}\left[  \sup\limits_{t\in\lbrack0,T]}\left\vert X_{1}%
(t)\right\vert ^{4}+\left(  \int_{E_{\epsilon}}\left\vert \delta
b(t,\Delta)+\delta\sigma(t,\Delta)+\delta g(t,\Delta)\right\vert dt\right)
^{2}\right] \\
\text{ \ }+C\mathbb{E}\left[  \sup\limits_{t\in\lbrack0,T]}\left\vert
X_{1}(t)\right\vert ^{2}\int_{E_{\epsilon}}\left\vert \delta\sigma
_{x}(t,\Delta)+\delta\sigma_{y}(t,\Delta)+\delta\sigma_{z}(t,\Delta
)\right\vert ^{2}dt\right] \\
\leq C\mathbb{E}\left[  \sup\limits_{t\in\lbrack0,T]}\left\vert X_{1}%
(t)\right\vert ^{4}\right]  +C\epsilon\mathbb{E}\left[  \int_{E_{\epsilon}%
}(1+|\bar{X}(t)|^{2}+|\bar{Y}(t)|^{2}+|\bar{Z}(t)|^{2}+|u(t)|^{2}+|\bar
{u}(t)|^{2})dt\right] \\
\text{ \ }+C\mathbb{E}\left[  \sup\limits_{t\in\lbrack0,T]}\left\vert
X_{1}(t)\right\vert ^{2}\int_{E_{\epsilon}}\left\vert \delta\sigma
_{x}(t,\Delta)+\delta\sigma_{y}(t,\Delta)+\delta\sigma_{z}(t,\Delta
)\right\vert ^{2}dt\right] \\
\leq C\epsilon^{2}%
\end{array}
\]
and%
\[%
\begin{array}
[c]{l}%
\mathbb{E}\left[  \sup\limits_{t\in\lbrack0,T]}(|X_{2}(t)|^{\beta}%
+|Y_{2}(t)|^{\beta})\right]  +\mathbb{E}\left[  \left(  \int_{0}^{T}%
|Z_{2}(t)|^{2}dt\right)  ^{\frac{\beta}{2}}\right] \\
\leq C\mathbb{E}\left[  \left\vert \varsigma\right\vert ^{\beta}+\left(
\int_{0}^{T}\left(  \left\vert L_{1}\left(  t\right)  \right\vert +\left\vert
L_{3}\left(  t\right)  \right\vert \right)  dt\right)  ^{\beta}+\left(
\int_{0}^{T}\left\vert L_{2}\left(  t\right)  \right\vert ^{2}dt\right)
^{\frac{\beta}{2}}\right] \\
\leq C\mathbb{E}\left[  \sup\limits_{t\in\lbrack0,T]}\left\vert X_{1}%
(t)\right\vert ^{2\beta}+\left(  \int_{E_{\epsilon}}\left\vert \delta
b(t,\Delta)+\delta\sigma(t,\Delta)+\delta g(t,\Delta)\right\vert dt\right)
^{\beta}\right] \\
\text{ \ }+C\mathbb{E}\left[  \sup\limits_{t\in\lbrack0,T]}\left\vert
X_{1}(t)\right\vert ^{\beta}\left(  \int_{E_{\epsilon}}\left\vert \delta
\sigma_{x}(t,\Delta)+\delta\sigma_{y}(t,\Delta)+\delta\sigma_{z}%
(t,\Delta)\right\vert ^{2}dt\right)  ^{\beta/2}\right] \\
\leq C\mathbb{E}\left[  \sup\limits_{t\in\lbrack0,T]}\left\vert X_{1}%
(t)\right\vert ^{2\beta}\right]  +C\epsilon^{\frac{\beta}{2}}\mathbb{E}\left[
\left(  \int_{E_{\epsilon}}(1+|\bar{X}(t)|^{2}+|\bar{Y}(t)|^{2}+|\bar
{Z}(t)|^{2}+|u(t)|^{2}+|\bar{u}(t)|^{2})dt\right)  ^{\frac{\beta}{2}}\right]
\\
\text{ \ }+C\mathbb{E}\left[  \sup\limits_{t\in\lbrack0,T]}\left\vert
X_{1}(t)\right\vert ^{\beta}\left(  \int_{E_{\epsilon}}\left\vert \delta
\sigma_{x}(t,\Delta)+\delta\sigma_{y}(t,\Delta)+\delta\sigma_{z}%
(t,\Delta)\right\vert ^{2}dt\right)  ^{\beta/2}\right] \\
=o\left(  \epsilon^{^{\beta/2}}\right)  .
\end{array}
\]

Now, we focus on the last estimate. We use the same notations $\xi
^{1,\epsilon}(t)$, $\eta^{1,\epsilon}(t)$, $\zeta^{1,\epsilon}(t)$,
$\xi^{2,\epsilon}(t)$, $\eta^{2,\epsilon}(t)$ and $\zeta^{2,\epsilon}(t)$ in
the proof of Lemma \ref{est-epsilon-bar} and Lemma \ref{est-one-order}. Let
\[%
\begin{array}
[c]{rl}%
\Theta(t,\Delta I_{E_{\epsilon}}) & =(\bar{X}(t),\bar{Y}(t),\bar{Z}%
(t)+\Delta(t)I_{E_{\epsilon}}(t));\text{ }\Theta^{\epsilon}(t):=(X^{\epsilon
}(t),Y^{\epsilon}(t),Z^{\epsilon}(t)).
\end{array}
\]
Define%
\[
\widetilde{D^{2}b^{\epsilon}}(t)=2\int_{0}^{1}\int_{0}^{1}\theta
D^{2}b(t,\Theta(t,\Delta I_{E_{\epsilon}})+\lambda\theta(\Theta^{\epsilon
}(t)-\Theta(t,\Delta I_{E_{\epsilon}})),u^{\epsilon}(t))d\theta d\lambda,
\]
and $\widetilde{D^{2}\sigma^{\epsilon}}(t)$, $\widetilde{D^{2}g^{\epsilon}%
}(t)$, $\tilde{\phi}_{xx}^{\epsilon}(T)$ are defined similarly. Then, we have
\begin{equation}
\left\{
\begin{array}
[c]{ll}%
d\xi^{3,\epsilon}(t)= & \left\{  b_{x}(t)\xi^{3,\epsilon}(t)+b_{y}%
(t)\eta^{3,\epsilon}(t)+b_{z}(t)\zeta^{3,\epsilon}(t)+A_{2}^{\epsilon
}(t)\right\}  dt\\
& +\left\{  \sigma_{x}(t)\xi^{3,\epsilon}(t)+\sigma_{y}(t)\eta^{3,\epsilon
}(t)+\sigma_{z}(t)\zeta^{3,\epsilon}(t)+B_{2}^{\epsilon}(t)\right\}  dB(t),\\
\xi^{3,\epsilon}(0)= & 0,
\end{array}
\right.  \label{x-x1-x2}%
\end{equation}
and%
\begin{equation}
\left\{
\begin{array}
[c]{lll}%
d\eta^{3,\epsilon}(t) & = & -\{\left\langle g_{x}(t),\xi^{3,\epsilon
}(t)\right\rangle +g_{y}(t)\eta^{3,\epsilon}(t)+g_{z}(t)\zeta^{3,\epsilon
}(t)+C_{2}^{\epsilon}(t)\}dt-\zeta^{3,\epsilon}(t)dB(t),\\
\eta^{3,\epsilon}(T) & = & \left\langle \phi_{x}(\bar{X}(T)),\xi^{3,\epsilon
}(T)\right\rangle +D_{2}^{\epsilon}(T),
\end{array}
\right.  \label{y-y1-y2}%
\end{equation}
where%
\[%
\begin{array}
[c]{ll}%
A_{2}^{\epsilon}(t)= & \left[  \delta b_{x}(t,\Delta)\xi^{1,\epsilon
}(t)+\delta b_{y}(t,\Delta)\eta^{1,\epsilon}(t)+\delta b_{z}(t,\Delta)\left(
\zeta^{1,\epsilon}(t)-\Delta(t)I_{E_{\epsilon}}(t)\right)  \right]
I_{E_{\epsilon}}(t)\\
& +\frac{1}{2}\widetilde{D^{2}b^{\epsilon}}(t)\left(  X_{1}(t)^{\intercal
},Y_{1}(t),\left\langle K_{1}(t),X_{1}(t)\right\rangle \right)  ^{2}-\frac
{1}{2}D^{2}b(t)\left(  X_{1}(t)^{\intercal},Y_{1}(t),\left\langle
K_{1}(t),X_{1}(t)\right\rangle \right)  ^{2},
\end{array}
\]%
\[%
\begin{array}
[c]{ll}%
B_{2}^{\epsilon}(t)= & \left[  \delta\sigma_{x}(t,\Delta)\xi^{2,\epsilon
}(t)+\delta\sigma_{y}(t,\Delta)\eta^{2,\epsilon}(t)+\delta\sigma_{z}%
(t,\Delta)\zeta^{2,\epsilon}(t)\right]  I_{E_{\epsilon}}(t)\\
& +\frac{1}{2}\widetilde{D^{2}\sigma^{\epsilon}}(t)\left(  X_{1}%
(t)^{\intercal},Y_{1}(t),\left\langle K_{1}(t),X_{1}(t)\right\rangle \right)
^{2}-\frac{1}{2}D^{2}\sigma(t)\left(  X_{1}(t)^{\intercal},Y_{1}%
(t),\left\langle K_{1}(t),X_{1}(t)\right\rangle \right)  ^{2},
\end{array}
\]%
\[%
\begin{array}
[c]{ll}%
C_{2}^{\epsilon}(t)= & \left[  \left\langle \delta g_{x}(t,\Delta
),\xi^{1,\epsilon}(t)\right\rangle +\delta g_{y}(t,\Delta)\eta^{1,\epsilon
}(t)+\delta g_{z}(t,\Delta)\left(  \zeta^{1,\epsilon}(t)-\Delta
(t)I_{E_{\epsilon}}(t)\right)  \right]  I_{E_{\epsilon}}(t)\\
& +\frac{1}{2}\left[  \xi^{1,\epsilon}(t)^{\intercal},\eta^{1,\epsilon
}(t),\zeta^{1,\epsilon}(t)-\Delta(t)I_{E_{\epsilon}}(t)\right]
\widetilde{D^{2}g^{\epsilon}}(t)\left[  \xi^{1,\epsilon}(t)^{\intercal}%
,\eta^{1,\epsilon}(t),\zeta^{1,\epsilon}(t)-\Delta(t)I_{E_{\epsilon}%
}(t)\right]  ^{\intercal}\\
& -\frac{1}{2}\left[  X_{1}(t)^{\intercal},Y_{1}(t),K_{1}(t)X_{1}(t)\right]
D^{2}g(t)\left[  X_{1}(t)^{\intercal},Y_{1}(t),K_{1}(t)X_{1}(t)\right]
^{\intercal},\\
D_{2}^{\epsilon}(T)= & \frac{1}{2}\left\langle \tilde{\phi}_{xx}^{\epsilon
}(T)\xi^{1,\epsilon}(T),\xi^{1,\epsilon}(T)\right\rangle -\frac{1}%
{2}\left\langle \phi_{xx}(\bar{X}(T))X_{1}(T),X_{1}(T)\right\rangle ,
\end{array}
\]
and $\widetilde{D^{2}b^{\epsilon}}(t)\left(  X_{1}(t)^{\intercal}%
,Y_{1}(t),\left\langle K_{1}(t),X_{1}(t)\right\rangle \right)  ^{2}$ is
defined similar to $D^{2}b(t)\left(  X_{1}(t)^{\intercal},Y_{1}%
(t),\left\langle K_{1}(t),X_{1}(t)\right\rangle \right)  ^{2}$.

By Lemma \ref{appen-th-linear-fbsde} in Appendix,
\[%
\begin{array}
[c]{rl}%
\eta^{3,\epsilon}(t)= & \left\langle p\left(  t\right)  ,\xi^{3,\epsilon
}(t)\right\rangle +\varphi\left(  t\right)  ,\\
\zeta^{3,\epsilon}(t)= & \left\langle K_{1}\left(  t\right)  ,\xi^{3,\epsilon
}(t)\right\rangle +\left(  1-\left\langle p(t),\sigma_{z}(t)\right\rangle
\right)  ^{-1}\left[  \left\langle p(t),\sigma_{y}(t)\right\rangle
\varphi(t)+\left\langle p(t),B_{2}^{\epsilon}(t)\right\rangle +\nu(t)\right]
.
\end{array}
\]
Then we have
\begin{equation}%
\begin{array}
[c]{ll}%
|\eta^{3,\epsilon}(0)| & =\left\vert \mathbb{E}\left[  \varphi\left(
0\right)  \right]  \right\vert \\
& \leq C\mathbb{E}\left[  \left\vert D_{2}^{\epsilon}(T)\right\vert +\int%
_{0}^{T}\left(  \left\vert A_{2}^{\epsilon}(t)\right\vert +\left\vert
B_{2}^{\epsilon}(t)\right\vert +\left\vert C_{2}^{\epsilon}(t)\right\vert
\right)  dt\right]  .
\end{array}
\label{second order-estimate}%
\end{equation}
We estimate each term as follows.

(1) {
\[%
\begin{array}
[c]{ll}%
\mathbb{E}\left[  |D_{2}^{\epsilon}(T)|\right]  & \leq C\left\{
\mathbb{E}\left[  |\tilde{\phi}_{xx}^{\epsilon}(T)-\phi_{xx}(\bar{X}%
(T))||\xi^{1,\epsilon}(T)|^{2}+|\xi^{2,\epsilon}(T)||\xi^{1,\epsilon}%
(T)+X_{1}(T)|\right]  \right\} \\
& =o(\epsilon).
\end{array}
\]
}

(2) {We estimate
\begin{equation}
\mathbb{E}\left[  \int_{0}^{T}|A_{2}^{\epsilon}(t)|dt\right]  =o(\epsilon).
\label{second order-part estimate}%
\end{equation}
Indeed, (\ref{second order-part estimate}) is due to the following estimates:
\[%
\begin{array}
[c]{l}%
\mathbb{E}\left[  \int_{0}^{T}|\delta b_{z}(t,\Delta)(\zeta^{1,\epsilon
}(t)-\Delta(t)I_{E_{\epsilon}}(t))|I_{E_{\epsilon}}(t)dt\right] \\
\leq\mathbb{E}\left[  \int_{E_{\epsilon}}|\delta b_{z}(t,\Delta)|\left(
|\zeta^{2,\epsilon}(t)|+|\left\langle K_{1}(t),X_{1}(t)\right\rangle |\right)
dt\right] \\
\leq C\mathbb{E}\left[  \int_{E_{\epsilon}}|\zeta^{2,\epsilon}(t)|dt\right]
+C\mathbb{E}\left[  \sup\limits_{t\in\lbrack0,T]}|X_{1}(t)|\int_{E_{\epsilon}%
}|\delta b_{z}(t,\Delta)|dt\right] \\
\leq C\epsilon^{\frac{1}{2}}\left\{  \mathbb{E}\left[  \int_{0}^{T}%
|\zeta^{2,\epsilon}(t)|^{2}dt\right]  \right\}  ^{\frac{1}{2}}+C\epsilon
\mathbb{E}[\sup\limits_{t\in\lbrack0,T]}|X_{1}(t)|]\\
=o(\epsilon),
\end{array}
\]%
\begin{equation}%
\begin{array}
[c]{l}%
\mathbb{E}\left[  \int_{0}^{T}\left\vert \widetilde{b}_{zz}^{i,\epsilon
}(t)\left(  \zeta^{1,\epsilon}(t)-\Delta(t)I_{E_{\epsilon}}(t)\right)
^{2}-b_{zz}^{i}(t)\left\langle K_{1}(t),X_{1}(t)\right\rangle ^{2}\right\vert
dt\right] \\
\leq\mathbb{E}\left[  \int_{0}^{T}\left\vert \widetilde{b}_{zz}^{i,\epsilon
}(t)\zeta^{2,\epsilon}(t)\left(  \zeta^{1,\epsilon}(t)-\Delta(t)I_{E_{\epsilon
}}(t)+K_{1}(t)X_{1}(t)\right)  \right\vert dt\right] \\
\text{ \ }+\mathbb{E}\left[  \int_{0}^{T}\left\vert \left(  \widetilde{b}%
_{zz}^{i,\epsilon}(t)-b_{zz}^{i}(t)\right)  \left\langle K_{1}(t),X_{1}%
(t)\right\rangle ^{2}\right\vert dt\right] \\
\leq C\left\{  \mathbb{E}\left[  \int_{0}^{T}\left\vert \zeta^{2,\epsilon
}(t)\right\vert ^{2}dt\right]  \right\}  ^{\frac{1}{2}}\left\{  \mathbb{E}%
\left[  \int_{0}^{T}\left\vert \zeta^{1,\epsilon}(t)-\Delta(t)I_{E_{\epsilon}%
}(t)+\left\langle K_{1}(t),X_{1}(t)\right\rangle \right\vert ^{2}dt\right]
\right\}  ^{\frac{1}{2}}\\
\text{ \ }+C\mathbb{E}\left[  \sup\limits_{t\in\lbrack0,T]}|X_{1}(t)|^{2}%
\int_{0}^{T}\left\vert \left(  \widetilde{b}_{zz}^{i,\epsilon}(t)-b_{zz}%
^{i}(t)\right)  \right\vert dt\right] \\
=o(\epsilon).
\end{array}
\label{est-d2b}%
\end{equation}
The other terms are similar.}

(3){ }The estimate of $\mathbb{E}\left[  \int_{0}^{T}|B_{2}^{\epsilon
}(t)|dt\right]  $:
\[%
\begin{array}
[c]{l}%
\mathbb{E}\left[  \int_{0}^{T}\left\vert \delta\sigma_{z}(t,\Delta
)\zeta^{2,\epsilon}(t)I_{E_{\epsilon}}(t)\right\vert dt\right] \\
\leq C\mathbb{E}\left[  \int_{E_{\epsilon}}|\zeta^{2,\epsilon}(t)|dt\right] \\
\leq C\epsilon^{\frac{1}{2}}\left\{  \mathbb{E}\left[  \int_{0}^{T}%
|\zeta^{2,\epsilon}(t)|^{2}dt\right]  \right\}  ^{\frac{1}{2}}\\
=o(\epsilon),
\end{array}
\]%
\[%
\begin{array}
[c]{l}%
\mathbb{E}\left[  \int_{0}^{T}\left\vert \tilde{\sigma}i_{zz}^{i,\epsilon
}(t)\left(  \zeta^{1,\epsilon}(t)-\Delta(t)I_{E_{\epsilon}}(t)\right)
^{2}-\sigma_{zz}^{i}(t)\left\langle K_{1}(t),X_{1}(t)\right\rangle
^{2}\right\vert dt\right] \\
\leq\mathbb{E}\left[  \int_{0}^{T}\left\vert \tilde{\sigma}_{zz}^{i,\epsilon
}(t)\left(  \zeta^{1,\epsilon}(t)-\Delta(t)I_{E_{\epsilon}}(t)+\left\langle
K_{1}(t),X_{1}(t)\right\rangle \right)  \zeta^{2,\epsilon}(t)\right\vert
dt\right] \\
\text{ \ }+\mathbb{E}\left[  \int_{0}^{T}\left\vert \tilde{\sigma}%
_{zz}^{i,\epsilon}(t)-\sigma_{zz}^{i}(t)\right\vert \left\langle
K_{1}(t),X_{1}(t)\right\rangle ^{2}dt\right] \\
\leq\mathbb{E}\left[  \int_{0}^{T}\left\vert \tilde{\sigma}_{zz}^{i,\epsilon
}(t)\left(  \zeta^{1,\epsilon}(t)-\Delta(t)I_{E_{\epsilon}}(t)\right)
\zeta^{2,\epsilon}(t)\right\vert dt\right]  +\mathbb{E}\left[  \int_{0}%
^{T}\left\vert \tilde{\sigma}_{zz}^{i,\epsilon}(t)\left\langle K_{1}%
(t),X_{1}(t)\right\rangle \zeta^{2,\epsilon}(t)\right\vert dt\right]
+o(\epsilon)\\
\leq C\left\{  \mathbb{E}\left[  \int_{0}^{T}\left\vert \left(  \zeta
^{1,\epsilon}(t)-\Delta(t)I_{E_{\epsilon}}(t)\right)  \right\vert
^{2}dt\right]  \right\}  ^{\frac{1}{2}}\left\{  \mathbb{E}\left[  \int_{0}%
^{T}\left\vert \zeta^{2,\epsilon}(t)\right\vert ^{2}dt\right]  \right\}
^{\frac{1}{2}}\\
\text{ \ }+C\mathbb{E}\left[  \sup\limits_{t\in\lbrack0,T]}\left\vert
X_{1}(t)\right\vert \int_{0}^{T}\left\vert \zeta^{2,\epsilon}(t)\right\vert
dt\right]  +o(\epsilon)\\
=o(\epsilon).
\end{array}
\]
{The other terms are similar.}

(4) {The estimate of $\mathbb{E}\left[  \int_{0}^{T}|C_{2}^{\epsilon
}(t)|dt\right]  $ is the same as the one of $\mathbb{E}\left[  \int_{0}%
^{T}|A_{2}^{\epsilon}(t)|dt\right]  $. }

Finally, we obtain
\[
Y^{\epsilon}(0)-\bar{Y}(0)-Y_{1}(0)-Y_{2}(0)=o(\epsilon).
\]
The proof is complete.
\end{proof}

In the above lemma, we only prove $Y^{\epsilon}(0)-\bar{Y}(0)-Y_{1}%
(0)-Y_{2}(0)=o(\epsilon)$ and have not deduced
\[
\mathbb{E}[\sup\limits_{t\in\lbrack0,T]}|Y^{\epsilon}(t)-\bar{Y}%
(t)-Y_{1}(t)-Y_{2}(t)|^{2}]=o(\epsilon^{2}).
\]
The reason is
\[
\mathbb{E}\left[  \int_{0}^{T}\left\vert \tilde{\sigma}_{zz}^{\epsilon
}(t)\left(  \zeta^{1,\epsilon}(t)-\Delta(t)I_{E_{\epsilon}}(t)\right)
\right\vert ^{2}\left\vert \zeta^{2,\epsilon}(t)\right\vert ^{2}dt\right]
=o(\epsilon^{2})
\]
may be not hold. But if
\begin{equation}
\sigma(t,x,y,z,u)=A(t)z+\sigma_{1}(t,x,y,u) \label{xigma-zz}%
\end{equation}
where $A(t)$ is a bounded adapted process, then $\sigma_{zz}\equiv0.$ In this
case, we can prove the following estimates.

\begin{lemma}
\label{lemma-est-sup}Under the same Assumptions as in Lemma
\ref{est-second-order}, and $\sigma(t,x,y,z,u)=A(t)z+$ $\sigma_{1}(t,x,y,u)$
where $A(t)$ is a bounded adapted process. Then%
\[%
\begin{array}
[c]{rl}%
\mathbb{E}\left[  \sup\limits_{t\in\lbrack0,T]}|X^{\epsilon}(t)-\bar
{X}(t)-X_{1}(t)-X_{2}(t)|^{2}\right]  & =o(\epsilon^{2}),\\
\mathbb{E}\left[  \sup\limits_{t\in\lbrack0,T]}|Y^{\epsilon}(t)-\bar
{Y}(t)-Y_{1}(t)-Y_{2}(t)|^{2}+\int_{0}^{T}|Z^{\epsilon}(t)-\bar{Z}%
(t)-Z_{1}(t)-Z_{2}(t)|^{2}dt\right]  & =o(\epsilon^{2}).
\end{array}
\]

\end{lemma}

\begin{proof}
We use all notations in Lemma \ref{est-second-order}. By Lemma \ref{est-l} in
Appendix, we have%
\[%
\begin{array}
[c]{l}%
\mathbb{E}\left[  \sup\limits_{t\in\lbrack0,T]}(|\xi^{3,\epsilon}%
(t)|^{2}+|\eta^{3,\epsilon}(t)|^{2})+\int_{0}^{T}|\zeta^{3,\epsilon}%
(t)|^{2}dt\right] \\
\leq C\mathbb{E}\left[  \left(  \int_{0}^{T}|A_{2}^{\epsilon}(t)|dt\right)
^{2}+\left(  \int_{0}^{T}|C_{2}^{\epsilon}(t)|dt\right)  ^{2}+\int_{0}%
^{T}|B_{2}^{\epsilon}(t)|^{2}dt+|D_{2}^{\epsilon}(T)|^{2}\right]  ,
\end{array}
\]
where $A_{2}^{\epsilon}(\cdot)$, $C_{2}^{\epsilon}(\cdot)$, $D_{2}^{\epsilon
}(T)$ are the same as Lemma \ref{est-second-order}, and
\[%
\begin{array}
[c]{ll}%
B_{2}^{\epsilon}(t)= & \left[  \delta\sigma_{x}(t)\xi^{2,\epsilon}%
(t)+\delta\sigma_{y}(t)\eta^{2,\epsilon}(t)\right]  I_{E_{\epsilon}}%
(t)+\frac{1}{2}\widetilde{D^{2}\sigma^{\epsilon}}(t)\left(  \xi^{1,\epsilon
}(t)^{\intercal},\eta^{1,\epsilon}(t)\right)  ^{2}\\
& -\frac{1}{2}D^{2}\sigma(t)\left(  X_{1}(t)^{\intercal},Y_{1}(t)\right)
^{2}.
\end{array}
\]
Combing Lemmas 3.15 and 3.16 in \cite{Hu-JX}, we can obtain the desired estimates.
\end{proof}

\subsection{Maximum principle}

\label{section-mp}Note that $Y_{1}(0)=0$, by Lemma \ref{est-second-order}, we
have
\[
J(u^{\epsilon}(\cdot))-J(\bar{u}(\cdot))=Y^{\epsilon}(0)-\bar{Y}%
(0)=Y_{2}(0)+o(\epsilon).
\]
In order to obtain $Y_{2}(0)$, we introduce the following second-order adjoint
equation:
\begin{equation}
\left\{
\begin{array}
[c]{l}%
-dP(t)\\
=\left\{  \left(  D\sigma(t)[I_{n\times n},p(t),K_{1}(t)]^{\intercal}\right)
^{\intercal}P(t)D\sigma(t)[I_{n\times n},p(t),K_{1}(t)]^{\intercal
}+P(t)Db(t)[I_{n\times n},p(t),K_{1}(t)]^{\intercal}\right. \\
+\left(  Db(t)[I_{n\times n},p(t),K_{1}(t)]^{\intercal}\right)  ^{\intercal
}P(t)+P(t)H_{y}(t)+Q(t)D\sigma(t)[I_{n\times n},p(t),K_{1}(t)]^{\intercal}\\
\left.  +\left(  D\sigma(t)[I_{n\times n},p(t),K_{1}(t)]^{\intercal}\right)
^{\intercal}Q(t)+\left[  I_{n\times n},p(t),K_{1}(t)\right]  D^{2}H(t)\left[
I_{n\times n},p(t),K_{1}(t)\right]  ^{\intercal}+H_{z}(t)K_{2}(t)\right\}
dt\\
-Q(t)dB(t),\\
P(T)=\phi_{xx}(\bar{X}(T)),
\end{array}
\right.  \label{eq-P}%
\end{equation}
where
\[%
\begin{array}
[c]{ll}%
H(t,x,y,z,u,p,q)= & g(t,x,y,z,u)+\left\langle p,b(t,x,y,z,u)\right\rangle
+\left\langle q,\sigma(t,x,y,z,u)\right\rangle ,
\end{array}
\]%
\[%
\begin{array}
[c]{ll}%
K_{2}(t)= & (1-\left\langle p(t),\sigma_{z}(t)\right\rangle )^{-1}\left\{
\sigma_{y}(t)p(t)^{\intercal}P(t)+\left(  \sigma_{x}(t)+\sigma_{y}%
(t)p(t)^{\intercal}+\sigma_{z}(t)K_{1}(t)^{\intercal}\right)  ^{\intercal
}P(t)\right\} \\
& +(1-\left\langle p(t),\sigma_{z}(t)\right\rangle )^{-1}\left\{  P(t)\left(
\sigma_{x}(t)+\sigma_{y}(t)p(t)^{\intercal}+\sigma_{z}(t)K_{1}(t)^{\intercal
}\right)  +Q(t)+p(t)D^{2}\sigma(t)\left(  I_{n\times n},p(t),K_{1}(t)\right)
^{2}\right\}  ,
\end{array}
\]
and $p(t)D^{2}\sigma(t)\left(  I_{n\times n},p(t),K_{1}(t)\right)  ^{2}%
\in\mathbb{R}^{n\times n}$ such that
\[
\left\langle p(t)D^{2}\sigma(t)\left(  I_{n\times n},p(t),K_{1}(t)\right)
^{2}X_{1}(t),X_{1}(t)\right\rangle =\left\langle p(t),D^{2}\sigma(t)\left(
X_{1}(t)^{\intercal},Y_{1}(t),\left\langle K_{1}(t),X_{1}(t)\right\rangle
\right)  ^{2}\right\rangle ,
\]
$DH(t)$, $D^{2}H(t)$ are defined similar to $D\psi$ and $D^{2}\psi$.

(\ref{eq-P}) is a linear BSDE with uniformly Lipschitz continuous coefficients
and it has a unique solution. Before we deduce the relationship between
$X_{2}(\cdot)$ and $(Y_{2}(\cdot),Z_{2}(\cdot))$, we introduce the following
equation:
\begin{equation}%
\begin{array}
[c]{rl}%
\hat{Y}(t)= & \int_{t}^{T}\left\{  (H_{y}(s)+g_{z}(s)\left\langle \sigma
_{y}(s),p(s)\right\rangle (1-\left\langle p(s),\sigma_{z}(s)\right\rangle
)^{-1})\hat{Y}(s)\right. \\
& \text{ \ }+\left(  H_{z}(s)+g_{z}(s)\left\langle \sigma_{z}%
(s),p(s)\right\rangle (1-\left\langle p(s),\sigma_{z}(s)\right\rangle
)^{-1}\right)  \hat{Z}(s)\\
& \text{ \ }\left.  +\left[  \delta H(s,\Delta)+\frac{1}{2}\delta
\sigma(s,\Delta)^{\intercal}P(s)\delta\sigma(s,\Delta)\right]  I_{E_{\epsilon
}}(s)\right\}  ds-\int_{t}^{T}\hat{Z}(s)dB(s),
\end{array}
\label{eq-y-hat}%
\end{equation}
where $\delta H(s,\Delta):=\left\langle p(s),\delta b(s,\Delta)\right\rangle
+\left\langle q(s),\delta\sigma(s,\Delta)\right\rangle +\delta g(s,\Delta)$.
It is also a linear BSDE and has a unique solution.

\begin{lemma}
\label{relation-y2} Under the same Assumptions as in Lemma
\ref{est-second-order}. Then we have%
\[%
\begin{array}
[c]{rl}%
Y_{2}(t) & =\left\langle p(t),X_{2}(t)\right\rangle +\frac{1}{2}\left\langle
P(t)X_{1}(t),X_{1}(t)\right\rangle +\hat{Y}(t),\\
Z_{2}(t) & =\mathbf{I(t)}+\hat{Z}(t),
\end{array}
\]
where $(\hat{Y}(\cdot),\hat{Z}(\cdot))$ is the solution to (\ref{eq-y-hat})
and
\begin{align*}
\mathbf{I(t)}  &  =\left\langle K_{1}(t),X_{2}(t)\right\rangle +\frac{1}%
{2}\left\langle K_{2}(t)X_{1}(t),X_{1}(t)\right\rangle +(1-\left\langle
p(t),\sigma_{z}(t)\right\rangle )^{-1}\left\langle p(t),\sigma_{y}(t)\hat
{Y}(t)+\sigma_{z}(t)\hat{Z}(t)\right\rangle \\
&  \text{ }\;+\left\langle P(t)\delta\sigma(t,\Delta),X_{1}(t)\right\rangle
I_{E_{\epsilon}}(t)+(1-\left\langle p(t),\sigma_{z}(t)\right\rangle
)^{-1}\left\langle p(t),\delta\sigma_{x}(t,\Delta)X_{1}(t)\right\rangle
I_{E_{\epsilon}}(t)\\
&  \;\ +(1-\left\langle p(t),\sigma_{z}(t)\right\rangle )^{-1}\left[
\left\langle p(t),\delta\sigma_{y}(t,\Delta)\left\langle p(t),X_{1}%
(t)\right\rangle +\delta\sigma_{z}(t,\Delta)\left\langle K_{1}(t),X_{1}%
(t)\right\rangle \right\rangle \right]  I_{E_{\epsilon}}(t).
\end{align*}

\end{lemma}

\begin{proof}
Using the same method as in Lemma \ref{lemma-y1}, we can deduce the above
relationship similarly.
\end{proof}

Consider the following equation:
\begin{equation}
\left\{
\begin{array}
[c]{rl}%
d\gamma(t)= & \gamma(t)\left[  H_{y}(t)+(1-\left\langle p(t),\sigma
_{z}(t)\right\rangle )^{-1}g_{z}(t)\left\langle p(t),\sigma_{y}%
(t)\right\rangle \right]  dt\\
& +\gamma(t)\left[  H_{z}(t)+(1-\left\langle p(t),\sigma_{z}(t)\right\rangle
)^{-1}g_{z}(t)\left\langle p(t),\sigma_{z}(t)\right\rangle \right]  dB(t),\\
\gamma(0)= & 1.
\end{array}
\right.  \label{eq-gamma}%
\end{equation}
Applying It\^{o}'s formula to $\gamma(t)\hat{Y}(t)$, we obtain
\[%
\begin{array}
[c]{rl}%
\hat{Y}(0)= & \mathbb{E}\left\{  \int_{0}^{T}\gamma(t)\left[  \delta
H(t,\Delta)+\frac{1}{2}\delta\sigma(s,\Delta)^{\intercal}P(s)\delta
\sigma(s,\Delta)\right]  I_{E_{\epsilon}}(t)dt\right\}  .
\end{array}
\]
Define
\begin{equation}%
\begin{array}
[c]{l}%
\mathcal{H}(t,x,y,z,u,p,q,P)\\
=\left\langle p,b(t,x,y,z+\Delta(t),u)\right\rangle +\left\langle
q,\sigma(t,x,y,z+\Delta(t),u)\right\rangle +g(t,x,y,z+\Delta(t),u)\\
\text{ }+\frac{1}{2}(\sigma(t,x,y,z+\Delta(t),u)-\sigma(t,\bar{X}(t),\bar
{Y}(t),\bar{Z}(t),\bar{u}(t)))^{\intercal}P(\sigma(t,x,y,z+\Delta
(t),u)-\sigma(t,\bar{X}(t),\bar{Y}(t),\bar{Z}(t),\bar{u}(t))),
\end{array}
\label{def-H}%
\end{equation}
where $\Delta(t)$ is defined in (\ref{def-delt}) corresponding to $u(t)=u$. It
is easy to check that
\begin{align*}
&  \delta H(t,\Delta)+\frac{1}{2}\delta\sigma(t,\Delta)^{\intercal}%
P(t)\delta\sigma(t,\Delta)\\
&  =\mathcal{H}(t,\bar{X}(t),\bar{Y}(t),\bar{Z}%
(t),u(t),p(t),q(t),P(t))-\mathcal{H}(t,\bar{X}(t),\bar{Y}(t),\bar{Z}%
(t),\bar{u}(t),p(t),q(t),P(t)).
\end{align*}
Noting that $\gamma(t)>0$ for $t\in\lbrack0,T]$, then we obtain the following
maximum principle.

\begin{theorem}
\label{Th-MP}Under the same Assumptions as in Lemma \ref{est-second-order}.
Let $\bar{u}(\cdot)\in\mathcal{U}[0,T]$ be optimal and $(\bar{X}(\cdot
),\bar{Y}(\cdot),\bar{Z}(\cdot))$ be the corresponding state processes of
(\ref{state-eq}). Then the following stochastic maximum principle holds:
\begin{equation}
\mathcal{H}(t,\bar{X}(t),\bar{Y}(t),\bar{Z}(t),u,p(t),q(t),P(t))\geq
\mathcal{H}(t,\bar{X}(t),\bar{Y}(t),\bar{Z}(t),\bar{u}%
(t),p(t),q(t),P(t)),\ \ \ \forall u\in U\ a.e.,\ a.s., \label{mp-1}%
\end{equation}
where $(p\left(  \cdot\right)  ,q\left(  \cdot\right)  )$, $\left(  P\left(
\cdot\right)  ,Q\left(  \cdot\right)  \right)  $ satisfy (\ref{eq-p}),
(\ref{eq-P}) respectively, and $\Delta(\cdot)$ satisfies (\ref{def-delt}).
\end{theorem}

\section{The case when $q$ is unbounded}

In this section, we consider the case when $q$ is unbounded and propose the
second kind of assumptions.

The relations $Y_{1}(t)=\left\langle p(t),X_{1}(t)\right\rangle $ and
$\ Z_{1}(t)=\left\langle K_{1}(t),X_{1}(t)\right\rangle +\Delta
(t)I_{E_{\epsilon}}(t)$ in Lemma \ref{lemma-y1}, is the key point to derive
the maximum principle (\ref{mp-1}). Note that to prove Lemma \ref{lemma-y1},
we need Assumption \ref{assm-p-q}, which implies%
\begin{equation}
\mathbb{E}\left[  \sup\limits_{t\in\lbrack0,T]}|\tilde{X}_{1}(t)|^{2}\right]
<\infty. \label{eq-new211}%
\end{equation}
However, under the following assumption, combing Theorems
\ref{appen-th-linear-fbsde-unb} we can obtain the relations $Y_{1}%
(t)=\left\langle p(t),X_{1}(t)\right\rangle $ and$\ Z_{1}(t)=\left\langle
K_{1}(t),X_{1}(t)\right\rangle +\Delta(t)I_{E_{\epsilon}}(t)$ without the
Assumption $q(\cdot)$ is bounded.

\begin{assumption}
\label{assm-sig-small} $\sigma(t,x,y,z,u)=A(t)z+\sigma_{1}(t,x,y,u)$.
\end{assumption}

\begin{assumption}
\label{assm-ex-unb}For any $u^{\epsilon}(\cdot)\in\mathcal{U}[0,T]$ and
$\beta\in\lbrack2,8)$, the FBSDE (\ref{eq-lfb}) has a unique solution
$(\hat{X}(\cdot),\hat{Y}(\cdot),\hat{Z}(\cdot))\in L_{\mathcal{F}}^{\beta
}(\Omega;C([0,T],\mathbb{R}^{n}))\times L_{\mathcal{F}}^{\beta}(\Omega
;C([0,T],\mathbb{R}))\times L_{\mathcal{F}}^{2,\beta}([0,T];\mathbb{R})$.
Moreover, we assume that the following estimate for FBSDE (\ref{eq-lfb})
holds, that is,%
\[%
\begin{array}
[c]{l}%
||(\hat{X},\hat{Y},\hat{Z})||_{\beta}^{\beta}=\mathbb{E}\left\{
\sup\limits_{t\in\lbrack0,T]}\left[  |\hat{X}(t)|^{\beta}+|\hat{Y}(t)|^{\beta
}\right]  +\left(  \int_{0}^{T}|\hat{Z}(t)|^{2}dt\right)  ^{\frac{\beta}{2}%
}\right\} \\
\ \ \ \ \ \ \ \leq C\mathbb{E}\left\{  \left(  \int_{0}^{T}[|L_{1}%
(t)|+|L_{3}(t)|]dt\right)  ^{\beta}+\left(  \int_{0}^{T}|L_{2}(t)|^{2}%
dt\right)  ^{\frac{\beta}{2}}+|\varsigma|^{\beta}+|x_{0}|^{\beta}\right\}  ,
\end{array}
\]
where $C$ depends on $T$, $\beta$, $\left\Vert \psi_{x}\right\Vert _{\infty}%
$,$\left\Vert \psi_{y}\right\Vert _{\infty}$, $\left\Vert \psi_{z}\right\Vert
_{\infty}$, $c_{1}$. \label{est-fbsde-lp}
\end{assumption}

In this case, the first-order adjoint equation becomes
\begin{equation}
\left\{
\begin{array}
[c]{rl}%
dp(t)= & -\left\{  g_{x}(t)+g_{y}(t)p(t)+g_{z}(t)K_{1}(t)+b_{x}%
(t)p(t)+\left\langle b_{y}(t),p(t)\right\rangle p(t)+\left\langle
b_{z}(t),K_{1}(t)\right\rangle p(t)\right. \\
& \left.  +\sigma_{x}(t)q(t)+\left\langle \sigma_{y}(t),p(t)\right\rangle
q(t)+\left\langle A(t),K_{1}(t)\right\rangle q(t)\right\}  dt+q(t)dB(t),\\
p(T)= & \phi_{x}(\bar{X}(T)),
\end{array}
\right.  \label{eq-p-q-unb}%
\end{equation}
where
\[
K_{1}(t)=(1-\left\langle p(t),A(t)\right\rangle )^{-1}\left[  \sigma
_{x}(t)p(t)+\left\langle \sigma_{y}(t),p(t)\right\rangle p(t)+q(t)\right]  .
\]

\begin{assumption}
\label{assm-p-q-unb} Assume the BSDEs (\ref{eq-p-q-unb}) have a unique
solution $(p(\cdot),q(\cdot))\in L_{\mathcal{F}}^{\infty}(\Omega
;C([0,T],\mathbb{R}^{n}))\times L_{\mathcal{F}}^{2,2}([0,T];\mathbb{R}^{n})$
such that $\left\vert 1-\left\langle p(t),\gamma_{2}(t)\right\rangle
\right\vert ^{-1}$ is bounded.
\end{assumption}

The first-order variational equation becomes
\[
\left\{
\begin{array}
[c]{rl}%
dX_{1}(t)= & \left[  b_{x}(t)X_{1}(t)+b_{y}(t)Y_{1}(t)+b_{z}(t)(Z_{1}%
(t)-\Delta(t)I_{E_{\epsilon}}(t))\right]  dt\\
& +\left[  \sigma_{x}(t)X_{1}(t)+\sigma_{y}(t)Y_{1}(t)+A(t)(Z_{1}%
(t)-\Delta(t)I_{E_{\epsilon}}(t))+\delta\sigma(t,\Delta)I_{E_{\epsilon}%
}(t)\right]  dB(t),\\
X_{1}(0)= & 0,
\end{array}
\right.
\]
and
\[
\left\{
\begin{array}
[c]{lll}%
dY_{1}(t) & = & -\left[  \left\langle g_{x}(t),X_{1}(t)\right\rangle
+g_{y}(t)Y_{1}(t)+g_{z}(t)(Z_{1}(t)-\Delta(t)I_{E_{\epsilon}}(t))-\left\langle
q(t),\delta\sigma(t,\Delta)\right\rangle I_{E_{\epsilon}}(t)\right]
dt+Z_{1}(t)dB(t),\\
Y_{1}(T) & = & \phi_{x}(\bar{X}(T))X_{1}(T),
\end{array}
\right.
\]
where
\[
\Delta(t)=\left(  1-\left\langle p(t),A(t)\right\rangle \right)
^{-1}\left\langle p(t),\sigma_{1}(t,\bar{X}(t),\bar{Y}(t),u(t))-\sigma
_{1}(t,\bar{X}(t),\bar{Y}(t),\bar{u}(t))\right\rangle .
\]

\begin{assumption}
\label{assm-4}Suppose the following SDE
\begin{equation}
\left\{
\begin{array}
[c]{rl}%
d\tilde{X}_{1}(t)= & \left[  b_{x}(t)\tilde{X}_{1}(t)+b_{y}(t)\left\langle
p(t),\tilde{X}_{1}(t)\right\rangle +b_{z}(t)\left\langle K_{1}(t),\tilde
{X}_{1}(t)\right\rangle \right]  dt\\
& +\left[  \sigma_{x}(t)\tilde{X}_{1}(t)+\sigma_{y}(t)\left\langle
p(t),\tilde{X}_{1}(t)\right\rangle +A(t)\left\langle K_{1}(t),\tilde{X}%
_{1}(t)\right\rangle +\delta\sigma(t,\Delta)I_{E_{\epsilon}}(t)\right]
dB(t),\\
\tilde{X}_{1}(0)= & 0,
\end{array}
\right.  \label{eq-1214-1}%
\end{equation}
has a unique solution $\tilde{X}_{1}(\cdot)\in L_{\mathcal{F}}^{4}%
(\Omega;C([0,T],\mathbb{R}^{n}))$.
\end{assumption}

By Theorem \ref{appen-th-linear-fbsde-unb}, we have the following relationship.

\begin{lemma}
\label{est-one-order-q-unbound} Suppose that Assumptions \ref{assmlip},
\ref{assm-ex}, \ref{assm-delta}, \ref{assm-sig-small}, \ref{assm-ex-unb},
\ref{assm-p-q-unb} and \ref{assm-4} hold. Then we have
\begin{align*}
Y_{1}(t)  &  =\left\langle p(t),X_{1}(t)\right\rangle ,\\
Z_{1}(t)  &  =\left\langle K_{1}(t),X_{1}(t)\right\rangle +\Delta
(t)I_{E_{\epsilon}}(t),
\end{align*}
where $p(\cdot)$ is the solution of (\ref{eq-p-q-unb}).
\end{lemma}

\begin{lemma}
Under the same Assumptions as in Lemma \ref{est-one-order-q-unbound}, for any
$2\leq\beta<8$, we have the following estimates
\begin{equation}
\mathbb{E}\left[  \sup\limits_{t\in\lbrack0,T]}\left(  |X_{1}(t)|^{\beta
}+|Y_{1}(t)|^{\beta}\right)  \right]  +\mathbb{E}\left[  \left(  \int_{0}%
^{T}|Z_{1}(t)|^{2}dt\right)  ^{\beta/2}\right]  =O(\epsilon^{\beta/2}),
\label{est-x1-y1-q-unbound}%
\end{equation}%
\[%
\begin{array}
[c]{l}%
\mathbb{E}\left[  \sup\limits_{t\in\lbrack0,T]}\left(  |X^{\epsilon}%
(t)-\bar{X}(t)-X_{1}(t)|^{4}+|Y^{\epsilon}(t)-\bar{Y}(t)-Y_{1}(t)|^{4}\right)
\right] \\
+\mathbb{E}\left[  \left(  \int_{0}^{T}|Z^{\epsilon}(t)-\bar{Z}(t)-Z_{1}%
(t)|^{2}dt\right)  ^{2}\right]  =o(\epsilon^{2}).
\end{array}
\]

\end{lemma}

\begin{proof}
Applying the $L^{\beta}$-estimates for $(X_{1}(\cdot),Y_{1}(\cdot),Z_{1}%
(\cdot))$, $\left(  \xi^{2,\epsilon}(t),\eta^{2,\epsilon}(t),\zeta
^{2,\epsilon}(t)\right)  $ and following the same steps as Lemma 3.23 in
\cite{Hu-JX}, we can obtain the desired estimates.
\end{proof}

The second-order variational equation becomes
\begin{equation}
\left\{
\begin{array}
[c]{rl}%
dX_{2}(t)= & \left\{  b_{x}(t)X_{2}(t)+b_{y}(t)Y_{2}(t)+b_{z}(t)Z_{2}%
(t)+\delta b(t,\Delta)I_{E_{\epsilon}}(t)\right. \\
& \left.  +\frac{1}{2}D^{2}b(t)\left(  X_{1}(t)^{\intercal},Y_{1}%
(t),\left\langle K_{1}(t),X_{1}(t)\right\rangle \right)  ^{2}\right\}  dt\\
& +\left\{  \sigma_{x}(t)X_{2}(t)+\sigma_{y}(t)Y_{2}(t)+A(t)Z_{2}(t)+\left[
\delta\sigma_{x}(t)X_{1}(t)+\delta\sigma_{y}(t)Y_{1}(t)\right]  I_{E_{\epsilon
}}(t)\right. \\
& \left.  +\frac{1}{2}D^{2}\sigma_{1}(t)\left(  X_{1}(t)^{\intercal}%
,Y_{1}(t)\right)  ^{2}\right\}  dB(t),\\
X_{2}(0)= & 0,
\end{array}
\right.  \label{new-form-x2-q-unbound}%
\end{equation}%
\begin{equation}
\left\{
\begin{array}
[c]{ll}%
dY_{2}(t)= & -\left\{  \left\langle g_{x}(t),X_{2}(t)\right\rangle
+g_{y}(t)Y_{2}(t)+g_{z}(t)Z_{2}(t)+\left\langle q(t),\delta\sigma
(t,\Delta)\right\rangle I_{E_{\epsilon}}(t)+\delta g(t,\Delta)I_{E_{\epsilon}%
}(t)\right. \\
& \left.  +\frac{1}{2}\left(  X_{1}(t)^{\intercal},Y_{1}(t),\left\langle
K_{1}(t),X_{1}(t)\right\rangle \right)  D^{2}g(t)\left(  X_{1}(t)^{\intercal
},Y_{1}(t),\left\langle K_{1}(t),X_{1}(t)\right\rangle \right)  ^{\intercal
}\right\}  dt+Z_{2}(t)dB(t),\\
Y_{2}(T)= & \left\langle \phi_{x}(\bar{X}(T)),X_{2}(T)\right\rangle +\frac
{1}{2}\left\langle \phi_{xx}(\bar{X}(T))X_{1}(T),X_{1}(T)\right\rangle .
\end{array}
\right.  \label{new-form-y2-q-unbound}%
\end{equation}
The following second-order estimates hold.

\begin{lemma}
\label{est-second-order-q-unbound}Under the same Assumptions as in Lemma
\ref{est-one-order-q-unbound}, we have the following estimates%
\[%
\begin{array}
[c]{rl}%
\mathbb{E}\left[  \sup\limits_{t\in\lbrack0,T]}|X^{\epsilon}(t)-\bar
{X}(t)-X_{1}(t)-X_{2}(t)|^{2}\right]  & =o(\epsilon^{2}),\\
\mathbb{E}\left[  \sup\limits_{t\in\lbrack0,T]}|Y^{\epsilon}(t)-\bar
{Y}(t)-Y_{1}(t)-Y_{2}(t)|^{2}\right]  +\mathbb{E}\left[  \int_{0}%
^{T}|Z^{\epsilon}(t)-\bar{Z}(t)-Z_{1}(t)-Z_{2}(t)|^{2}dt\right]  &
=o(\epsilon^{2}).
\end{array}
\]

\end{lemma}

\begin{proof}
We use the same notations $A_{2}^{\epsilon}(t)$ $C_{2}^{\epsilon}(t)$ and
$D_{2}^{\epsilon}(T)$ as in\ Lemma \ref{est-second-order}. The only different
term is
\[%
\begin{array}
[c]{ll}%
B_{2}^{\epsilon}(t)= & \delta\sigma_{x}(t)\xi^{2,\epsilon}(t)I_{E_{\epsilon}%
}(t)+\delta\sigma_{y}(t)\eta^{2,\epsilon}(t)I_{E_{\epsilon}}(t)+\frac{1}%
{2}\widetilde{D^{2}\sigma^{\epsilon}}(t)\left(  \xi^{1,\epsilon}%
(t)^{\intercal},\eta^{1,\epsilon}(t)\right)  ^{2}\\
& -\frac{1}{2}D^{2}\sigma(t)\left(  X_{1}(t)^{\intercal},Y_{1}(t)\right)
^{2}.
\end{array}
\]
Then, we have that
\begin{equation}
\left\{
\begin{array}
[c]{rl}%
d\xi^{3,\epsilon}(t)= & \left[  b_{x}(t)\xi^{3,\epsilon}(t)+b_{y}%
(t)\eta^{3,\epsilon}(t)+b_{z}(t)\zeta^{3,\epsilon}(t)+A_{2}^{\epsilon
}(t)\right]  dt\ \\
& +\left[  \sigma_{x}(t)\xi^{3,\epsilon}(t)+\sigma_{y}(t)\eta^{3,\epsilon
}(t)+A(t)\zeta^{3,\epsilon}(t)+B_{2}^{\epsilon}(t)\right]  dB(t),\\
\xi^{3,\epsilon}(0)= & 0,
\end{array}
\right.  \label{x-x1-x2-bz-0}%
\end{equation}
and%
\begin{equation}
\left\{
\begin{array}
[c]{ll}%
d\eta^{3,\epsilon}(t)= & -\left[  \left\langle g_{x}(t),\xi^{3,\epsilon
}(t)\right\rangle +g_{y}(t)\eta^{3,\epsilon}(t)+g_{z}(t)\zeta^{3,\epsilon
}(t)+C_{2}^{\epsilon}(t)\right]  dt+\zeta^{3,\epsilon}(t)dB(t),\\
\eta^{3,\epsilon}(T)= & \left\langle \phi_{x}(\bar{X}(T)),\xi^{3,\epsilon
}(T)\right\rangle +D_{2}^{\epsilon}(T).
\end{array}
\right.  \label{y-y1-y2-bz-0}%
\end{equation}
By Assumption \ref{assm-ex-unb},
\[%
\begin{array}
[c]{l}%
\mathbb{E}\left[  \sup\limits_{t\in\lbrack0,T]}\left(  |\xi^{3,\epsilon
}(t)|^{2}+|\eta^{3,\epsilon}(t)|^{2}\right)  +\int_{0}^{T}|\zeta^{3,\epsilon
}(t)|^{2}dt\right] \\
\leq\mathbb{E}\left[  \left(  \int_{0}^{T}|A_{2}^{\epsilon}(t)|dt\right)
^{2}+\left(  \int_{0}^{T}|C_{2}^{\epsilon}(t)|dt\right)  ^{2}+\int_{0}%
^{T}|B_{2}^{\epsilon}(t)|^{2}dt+|D_{2}^{\epsilon}(T)|^{2}\right]  .
\end{array}
\]
We can estimate term by term by the same steps in Lemma 3.24 in \cite{Hu-JX}.
Thus completes the proof.
\end{proof}

Now we introduce the second-order adjoint equation:%
\begin{equation}
\left\{
\begin{array}
[c]{l}%
-dP(t)\\
=\left\{  \left(  D\sigma(t)[I_{n\times n},p(t),K_{1}(t)]^{\intercal}\right)
^{\intercal}P(t)D\sigma(t)[I_{n\times n},p(t),K_{1}(t)]^{\intercal
}+P(t)Db(t)[I_{n\times n},p(t),K_{1}(t)]^{\intercal}\right. \\
\text{ \ }+\left(  Db(t)[I_{n\times n},p(t),K_{1}(t)]^{\intercal}\right)
^{\intercal}P(t)+P(t)H_{y}(t)+Q(t)D\sigma(t)[I_{n\times n},p(t),K_{1}%
(t)]^{\intercal}\\
\text{ \ }\left.  +\left(  D\sigma(t)[I_{n\times n},p(t),K_{1}(t)]^{\intercal
}\right)  ^{\intercal}Q(t)+\left[  I_{n\times n},p(t),K_{1}(t)\right]
D^{2}H(t)\left[  I_{n\times n},p(t),K_{1}(t)\right]  ^{\intercal}%
+H_{z}(t)K_{2}(t)\right\}  dt\\
\text{ \ }-Q(t)dB(t),\\
P(T)=\phi_{xx}(\bar{X}(T)),
\end{array}
\right.  \label{eq-P-sigmaz0}%
\end{equation}

where
\[%
\begin{array}
[c]{ll}%
H(t,x,y,z,u,p,q)= & g(t,x,y,z,u)+\left\langle p,b(t,x,y,z,u)\right\rangle
+\left\langle q,\sigma(t,x,y,z,u)\right\rangle ,
\end{array}
\]%
\[%
\begin{array}
[c]{rl}%
K_{2}(t)= & (1-\left\langle p(t),A(t)\right\rangle )^{-1}\left\{  \sigma
_{y}(t)p(t)^{\intercal}P(t)+\left(  \sigma_{x}(t)+\sigma_{y}(t)p(t)^{\intercal
}+A(t)K_{1}(t)^{\intercal}\right)  ^{\intercal}P(t)+Q(t)\right\} \\
& +(1-\left\langle p(t),A(t)\right\rangle )^{-1}\left\{  P(t)\left(
\sigma_{x}(t)+\sigma_{y}(t)p(t)^{\intercal}+A(t)K_{1}(t)^{\intercal}\right)
+p(t)D^{2}\sigma(t)\left(  I_{n\times n},p(t),K_{1}(t)\right)  ^{2}\right\}  .
\end{array}
\]

\noindent(\ref{eq-P-sigmaz0}) is a linear BSDE with non-Lipschitz coefficient
for $P(\cdot)$. Then, (\ref{eq-P-sigmaz0}) has a unique pair of solution
according to Theorem 5.21 in \cite{Pardoux-book}. By the same analysis as in
Lemma \ref{relation-y2}, we introduce the following auxiliary equation:
\begin{equation}%
\begin{array}
[c]{rl}%
\hat{Y}(t)= & \int_{t}^{T}\left\{  (H_{y}(s)+g_{z}(s)\left\langle \sigma
_{y}(s),p(s)\right\rangle (1-\left\langle p(s),A(t)\right\rangle )^{-1}%
)\hat{Y}(s)\right. \\
& \text{ \ }+\left(  H_{z}(s)+g_{z}(s)\left\langle \sigma_{z}%
(s),p(s)\right\rangle (1-\left\langle p(s),A(t)\right\rangle )^{-1}\right)
\hat{Z}(s)\\
& \text{ \ }\left.  +\left[  \delta H(s,\Delta)+\frac{1}{2}\delta
\sigma(s,\Delta)^{\intercal}P(s)\delta\sigma(s,\Delta)\right]  I_{E_{\epsilon
}}(s)\right\}  ds-\int_{t}^{T}\hat{Z}(s)dB(s),
\end{array}
\label{yhat-sigmaz0}%
\end{equation}
where $\delta H(s,\Delta):=\left\langle p(s),\delta b(s,\Delta)\right\rangle
+\left\langle q(s),\delta\sigma(s,\Delta)\right\rangle +\delta g(s,\Delta)$.
We obtain the following relationship.

\begin{lemma}
\label{relation-second-order-q-unbound} Suppose the same Assumptions as in
Lemma \ref{est-one-order-q-unbound} hold. Furthermore, we suppose the
following SDE
\[
\left\{
\begin{array}
[c]{rl}%
dX_{2}(t)= & \left\{  b_{x}(t)X_{2}(t)+b_{y}(t)\left(  \left\langle
p(t),X_{2}(t)\right\rangle +\frac{1}{2}X_{1}(t)^{\intercal}P(t)X_{1}%
(t)+\hat{Y}(t)\right)  +b_{z}(t)\left(  \mathbf{I(t)}+\hat{Z}(t)\right)
\right. \\
& \left.  +\delta b(t,\Delta)I_{E_{\epsilon}}(t)+\frac{1}{2}D^{2}b(t)\left(
X_{1}(t)^{\intercal},Y_{1}(t),\left\langle K_{1}(t),X_{1}(t)\right\rangle
\right)  ^{2}\right\}  dt\\
& +\left\{  \sigma_{x}(t)X_{2}(t)+\sigma_{y}(t)\left(  \left\langle
p(t),X_{2}(t)\right\rangle +\frac{1}{2}X_{1}(t)^{\intercal}P(t)X_{1}%
(t)+\hat{Y}(t)\right)  +A(t)\left(  \mathbf{I(t)}+\hat{Z}(t)\right)  \right.
\\
& \left.  +\left[  \delta\sigma_{x}(t)X_{1}(t)+\delta\sigma_{y}(t)Y_{1}%
(t)\right]  I_{E_{\epsilon}}(t)+\frac{1}{2}D^{2}\sigma_{1}(t)\left(
X_{1}(t)^{\intercal},Y_{1}(t)\right)  ^{2}\right\}  dB(t),\\
X_{2}(0)= & 0,
\end{array}
\right.
\]
has a unique solution $X_{2}(\cdot)\in L_{\mathcal{F}}^{2}(\Omega
;C([0,T],\mathbb{R}^{n}))$ and $\left\langle p(t),X_{2}(t)\right\rangle
+\frac{1}{2}X_{1}(t)^{\intercal}P(t)X_{1}(t)+\hat{Y}(t)\in L_{\mathcal{F}}%
^{2}(\Omega;C([0,T],\mathbb{R}))$, $\mathbf{I(t)}+\hat{Z}(t)\in L_{\mathcal{F}%
}^{2,2}([0,T];\mathbb{R})$, where $(\hat{Y}(\cdot),\hat{Z}(\cdot))$ is the
solution to (\ref{yhat-sigmaz0}) and
\begin{align*}
\mathbf{I(t)}  &  =\left\langle K_{1}(t),X_{2}(t)\right\rangle +\frac{1}%
{2}\left\langle K_{2}(t)X_{1}(t),X_{1}(t)\right\rangle +(1-\left\langle
p(t),A(t)\right\rangle )^{-1}\left\langle p(t),\sigma_{y}(t)\hat
{Y}(t)+A(t)\hat{Z}(t)\right\rangle \\
&  \text{ }\;+\left\langle P(t)\delta\sigma(t,\Delta),X_{1}(t)\right\rangle
I_{E_{\epsilon}}(t)+(1-\left\langle p(t),\sigma_{z}(t)\right\rangle
)^{-1}\left\langle p(t),\delta\sigma_{x}(t,\Delta)X_{1}(t)\right\rangle
I_{E_{\epsilon}}(t)\\
&  \;\ +(1-\left\langle p(t),\sigma_{z}(t)\right\rangle )^{-1}\left[
\left\langle p(t),\delta\sigma_{y}(t,\Delta)\left\langle p(t),X_{1}%
(t)\right\rangle \right\rangle \right]  I_{E_{\epsilon}}(t).
\end{align*}
Then the solution to FBSDE (\ref{new-form-x2-q-unbound}%
)-(\ref{new-form-y2-q-unbound}) has the following relationship
\[%
\begin{array}
[c]{rl}%
Y_{2}(t) & =\left\langle p(t),X_{2}(t)\right\rangle +\frac{1}{2}%
X_{1}(t)^{\intercal}P(t)X_{1}(t)+\hat{Y}(t),\\
Z_{2}(t) & =\mathbf{I(t)}+\hat{Z}(t).
\end{array}
\]

\end{lemma}

\begin{proof}
Applying the techniques in Lemma \ref{lemma-y1}, we can deduce the above
relationship similarly.
\end{proof}

Combing the estimates in Lemma \ref{est-second-order-q-unbound} and the
relationship in Lemma \ref{relation-second-order-q-unbound}, we deduce that
\[
Y^{\epsilon}(0)-\bar{Y}(0)=Y_{1}(0)+Y_{2}(0)+o(\epsilon)=\hat{Y}%
(0)+o(\epsilon)\geq0.
\]
Define
\[%
\begin{array}
[c]{l}%
\mathcal{H}(t,x,y,z,u,p,q,P)\\
=\left\langle p,b(t,x,y,z+\Delta(t),u)\right\rangle +\left\langle
q,\sigma(t,x,y,z+\Delta(t),u)\right\rangle +g(t,x,y,z+\Delta(t),u)\\
\text{ }+\frac{1}{2}(\sigma(t,x,y,z+\Delta(t),u)-\sigma(t,\bar{X}(t),\bar
{Y}(t),\bar{Z}(t),\bar{u}(t)))^{\intercal}P(\sigma(t,x,y,z+\Delta
(t),u)-\sigma(t,\bar{X}(t),\bar{Y}(t),\bar{Z}(t),\bar{u}(t))).
\end{array}
\]
By the same analysis as in Theorem \ref{Th-MP}, we obtain the following
maximum principle.

\begin{theorem}
\label{th-mp-q-unboud}Under the same Assumptions as in Lemma
\ref{relation-second-order-q-unbound}. Let $\bar{u}(\cdot)\in\mathcal{U}[0,T]$
be optimal and $(\bar{X}(\cdot),\bar{Y}(\cdot),\bar{Z}(\cdot))$ be the
corresponding state processes of (\ref{state-eq}). Then the following
stochastic maximum principle holds:
\[
\mathcal{H}(t,\bar{X}(t),\bar{Y}(t),\bar{Z}(t),u,p(t),q(t),P(t))\geq
\mathcal{H}(t,\bar{X}(t),\bar{Y}(t),\bar{Z}(t),\bar{u}%
(t),p(t),q(t),P(t)),\ \ \ \forall u\in U\ a.e.,\ a.s..
\]

\end{theorem}

\section{Appendix}

\subsection{$L^{\beta}$-estimate for FBSDE}

We introduce the following lemmas. Consider the controlled forward-backward
stochastic differential equation
\begin{equation}
\left\{
\begin{array}
[c]{rl}%
d\hat{X}(t)= & \left[  \alpha_{1}(t)\hat{X}(t)+\beta_{1}(t)\hat{Y}%
(t)+\gamma_{1}(t)\hat{Z}(t)+L_{1}(t)\right]  dt+\left[  \alpha_{2}(t)\hat
{X}(t)\right. \\
& \left.  +\beta_{2}(t)\hat{Y}(t)+\gamma_{2}(t)\hat{Z}(t)+L_{2}(t)\right]
dB(t),\\
d\hat{Y}(t)= & -\left[  \left\langle \alpha_{3}(t),\hat{X}(t)\right\rangle
+\beta_{3}(t)\hat{Y}(t)+\gamma_{3}(t)\hat{Z}(t)+L_{3}(t)\right]  dt+\hat
{Z}(t)dB(t),\\
\hat{X}(0)= & x_{0},\ \hat{Y}(T)=\left\langle \kappa,\hat{X}(T)\right\rangle
+\varsigma,
\end{array}
\right.  \label{appen-eq-xyz}%
\end{equation}
where $\alpha_{i}(\cdot)$, $\beta_{i}(\cdot)$, $\gamma_{i}(\cdot)$, $i=1,2,3$,
are bounded adapted processes, $\alpha_{1}(\cdot)$, $\ \alpha_{2}(\cdot
)\in\mathbb{R}^{n\times n}$, $\alpha_{3}(\cdot)$, $\beta_{1}(\cdot)$,
$\beta_{2}(\cdot)$, $\gamma_{1}(\cdot)$, $\gamma_{2}(\cdot)\in\mathbb{R}^{n}$,
$\beta_{3}(t)$, $\gamma_{3}(\cdot)\in\mathbb{R}$, $L_{1}(\cdot)\in
L_{\mathcal{F}}^{\beta}([0,T];\mathbb{R}^{n})$, $L_{3}(\cdot)\in
L_{\mathcal{F}}^{\beta}([0,T];\mathbb{R})$, $L_{2}(\cdot)\in L_{\mathcal{F}%
}^{2,\beta}([0,T];\mathbb{R}^{n})$, $\varsigma\in L_{\mathcal{F}_{T}}^{\beta
}(\Omega;\mathbb{R}^{n})$ for some $\beta\in\lbrack2,8]$, $\kappa\in
\mathbb{R}^{n}$ is a $\mathcal{F}_{T}$-measurable random variable. Suppose
that the solution to (\ref{appen-eq-xyz}) has the following relationship
\[
\hat{Y}(t)=\left\langle p(t),\hat{X}(t)\right\rangle +\varphi(t),
\]
where $p(t)$, $\varphi(t)$ satisfies
\begin{equation}
\left\{
\begin{array}
[c]{rl}%
dp(t)= & -A(t)dt+q(t)dB(t),\\
p(T)= & \kappa,
\end{array}
\right.  \label{appen-eq-pq}%
\end{equation}%
\begin{equation}
\left\{
\begin{array}
[c]{rl}%
d\varphi(t)= & -C(t)dt+\nu(t)dB(t),\\
\varphi(T)= & \varsigma,
\end{array}
\right.  \label{appen-phi}%
\end{equation}
$A(t)$ and $C(t)$ will be determined later. Applying It\^{o}'s formula to
$\left\langle p(t),\hat{X}(t)\right\rangle +\varphi(t)$, we have
\[%
\begin{array}
[c]{l}%
d\left(  \left\langle p(t),\hat{X}(t)\right\rangle +\varphi(t)\right) \\
=\left\{  \left\langle p(t),\alpha_{1}(t)\hat{X}(t)+\beta_{1}(t)\hat
{Y}(t)+\gamma_{1}(t)\hat{Z}(t)+L_{1}(t)\right\rangle -\left\langle
A(t),\hat{X}(t)\right\rangle \right. \\
\ \ \left.  +\left\langle q(t),\alpha_{2}(t)\hat{X}(t)+\beta_{2}(t)\hat
{Y}(t)+\gamma_{2}(t)\hat{Z}(t)+L_{2}(t)\right\rangle -C(t)\right\}  dt\\
\ \ +\left\{  \left\langle p(t),\alpha_{2}(t)\hat{X}(t)+\beta_{2}(t)\hat
{Y}(t)+\gamma_{2}(t)\hat{Z}(t)+L_{2}(t)\right\rangle +\left\langle
q(t),\hat{X}(t)\right\rangle +\nu(t)\right\}  dB(t).
\end{array}
\]
Comparing with the equation satisfied by $\hat{Y}(t)$, one has
\begin{equation}
\hat{Z}(t)=\left\langle p(t),\alpha_{2}(t)\hat{X}(t)+\beta_{2}(t)\hat
{Y}(t)+\gamma_{2}(t)\hat{Z}(t)+L_{2}(t)\right\rangle +\left\langle
q(t),\hat{X}(t)\right\rangle +\nu(t), \label{appen-eq-z}%
\end{equation}%
\begin{equation}%
\begin{array}
[c]{l}%
-\left[  \left\langle \alpha_{3}(t),\hat{X}(t)\right\rangle +\beta_{3}%
(t)\hat{Y}(t)+\gamma_{3}(t)\hat{Z}(t)+L_{3}(t)\right] \\
=\left\langle p(t),\alpha_{1}(t)\hat{X}(t)+\beta_{1}(t)\hat{Y}(t)+\gamma
_{1}(t)\hat{Z}(t)+L_{1}(t)\right\rangle -\left\langle A(t),\hat{X}%
(t)\right\rangle \\
\ \ +\left\langle q(t),\alpha_{2}(t)\hat{X}(t)+\beta_{2}(t)\hat{Y}%
(t)+\gamma_{2}(t)\hat{Z}(t)+L_{2}(t)\right\rangle -C(t).
\end{array}
\label{appen-relation-generator}%
\end{equation}
From equation (\ref{appen-eq-z}), we have the form of $\hat{Z}(t)$ as
\[%
\begin{array}
[c]{rl}%
\hat{Z}(t)= & \left(  1-\left\langle p(t),\gamma_{2}(t)\right\rangle \right)
^{-1}\left[  \left\langle p(t),\alpha_{2}(t)\hat{X}(t)+\beta_{2}(t)\hat
{Y}(t)+L_{2}(t)\right\rangle +\left\langle q(t),\hat{X}(t)\right\rangle
+\nu(t)\right] \\
= & \left(  1-\left\langle p(t),\gamma_{2}(t)\right\rangle \right)
^{-1}\left[  \left\langle \alpha_{2}(t)^{\intercal}p(t)+\left\langle
p(t),\beta_{2}(t)\right\rangle p(t)+q(t),\hat{X}(t)\right\rangle \right. \\
& +\left\langle p(t),\beta_{2}(t)\right\rangle \varphi(t)+\left\langle
p(t),L_{2}(t)\right\rangle +\nu(t)].
\end{array}
\]
From the equation (\ref{appen-relation-generator}), and utilizing the form of
$\hat{Y}(t)$ and $\hat{Z}(t)$, we derive that%
\begin{equation}%
\begin{array}
[c]{rl}%
A(t)= & \alpha_{3}(t)+\beta_{3}(t)p(t)+\gamma_{3}(t)K_{1}(t)+\alpha
_{1}(t)^{\intercal}p(t)+\left\langle p(t),\beta_{1}(t)\right\rangle p(t)\\
& +\left\langle p(t),\gamma_{1}(t)\right\rangle K_{1}(t)+\alpha_{2}%
(t)^{\intercal}q(t)+\left\langle q(t),\beta_{2}(t)\right\rangle
p(t)+\left\langle q(t),\gamma_{2}(t)\right\rangle K_{1}(t),
\end{array}
\label{new-eq-111}%
\end{equation}
where
\[
K_{1}(t)=\left(  1-\left\langle p(t),\gamma_{2}(t)\right\rangle \right)
^{-1}\left[  \alpha_{2}(t)^{\intercal}p(t)+\left\langle p(t),\beta
_{2}(t)\right\rangle p(t)+q(t)\right]  ,
\]
and%
\begin{equation}%
\begin{array}
[c]{rl}%
C(t)= & \left[  \beta_{3}(t)\varphi(t)+\gamma_{3}(t)\left(  1-\left\langle
p(t),\gamma_{2}(t)\right\rangle \right)  ^{-1}\left[  \left\langle
p(t),\beta_{2}(t)\right\rangle \varphi(t)+\left\langle p(t),L_{2}%
(t)\right\rangle +\nu(t)\right]  +L_{3}(t)\right] \\
& +\left\langle p(t),\beta_{1}(t)\varphi(t)+\gamma_{1}(t)\left(
1-\left\langle p(t),\gamma_{2}(t)\right\rangle \right)  ^{-1}\left[
\left\langle p(t),\beta_{2}(t)\right\rangle \varphi(t)+\left\langle
p(t),L_{2}(t)\right\rangle +\nu(t)\right]  +L_{1}(t)\right\rangle \\
& +\left\langle q(t),\beta_{2}(t)\varphi(t)+\gamma_{2}(t)\left(
1-\left\langle p(t),\gamma_{2}(t)\right\rangle \right)  ^{-1}\left[
\left\langle p(t),\beta_{2}(t)\right\rangle \varphi(t)+\left\langle
p(t),L_{2}(t)\right\rangle +\nu(t)\right]  +L_{2}(t)\right\rangle .
\end{array}
\label{new-eq-112}%
\end{equation}

\begin{lemma}
\label{appen-th-linear-fbsde}Assume (\ref{appen-eq-pq}) has a unique solution
$(p(\cdot),q(\cdot))\in L_{\mathcal{F}}^{\infty}(\Omega;C([0,T],\mathbb{R}%
^{n}))\times L_{\mathcal{F}}^{\infty}([0,T];\mathbb{R}^{n})$ such that
$\left\vert 1-\left\langle p(t),\gamma_{2}(t)\right\rangle \right\vert ^{-1}$
is bounded. Then
\end{lemma}

(i) BSDE (\ref{appen-phi}) has a unique solution in $L_{\mathcal{F}}^{\beta
}(\Omega;C([0,T],\mathbb{R}))\times L_{\mathcal{F}}^{2,\beta}([0,T];\mathbb{R}%
)$;

(ii) FBSDE (\ref{appen-eq-xyz}) has a solution $(\tilde{X}(\cdot),\tilde
{Y}(\cdot),\tilde{Z}(\cdot))\in L_{\mathcal{F}}^{\beta}(\Omega
;C([0,T],\mathbb{R}^{n}))\times L_{\mathcal{F}}^{\beta}(\Omega
;C([0,T],\mathbb{R}))\times L_{\mathcal{F}}^{2,\beta}([0,T];\mathbb{R})$,
where $\tilde{X}(\cdot)$ is the solution to
\begin{equation}
\left\{
\begin{array}
[c]{rl}%
d\tilde{X}(t)= & \left\{  \alpha_{1}(t)\tilde{X}(t)+\beta_{1}(t)\left\langle
p(t),\tilde{X}(t)\right\rangle +\gamma_{1}(t)\left\langle K_{1}\left(
t\right)  ,\tilde{X}(t)\right\rangle +\beta_{1}(t)\varphi(t)+L_{1}(t)\right.
\\
& \left.  +\gamma_{1}(t)\left(  1-\left\langle p(t),\gamma_{2}(t)\right\rangle
\right)  ^{-1}\left[  \left\langle p(t),\beta_{2}(t)\right\rangle
\varphi(t)+\left\langle p(t),L_{2}(t)\right\rangle +\nu(t)\right]  \right\}
dt\\
& +\left\{  \alpha_{2}(t)\tilde{X}(t)+\beta_{2}(t)\left\langle p(t),\tilde
{X}(t)\right\rangle +\gamma_{2}(t)\left\langle K_{1}\left(  t\right)
,\tilde{X}(t)\right\rangle +\beta_{2}(t)\varphi(t)+L_{2}(t)\right. \\
& \left.  +\gamma_{2}(t)\left(  1-\left\langle p(t),\gamma_{2}(t)\right\rangle
\right)  ^{-1}\left[  \left\langle p(t),\beta_{2}(t)\right\rangle
\varphi(t)+\left\langle p(t),L_{2}(t)\right\rangle +\nu(t)\right]  \right\}
dB(t),\\
\tilde{X}(0)= & x_{0},
\end{array}
\right.  \label{appen-eq-x}%
\end{equation}
and%
\begin{equation}%
\begin{array}
[c]{rl}%
\tilde{Y}(t)= & \left\langle p(t),\tilde{X}(t)\right\rangle +\varphi(t),\\
\tilde{Z}(t)= & \left\langle K_{1}\left(  t\right)  ,\tilde{X}(t)\right\rangle
+\left(  1-\left\langle p(t),\gamma_{2}(t)\right\rangle \right)  ^{-1}\left[
\left\langle p(t),\beta_{2}(t)\right\rangle \varphi(t)+\left\langle
p(t),L_{2}(t)\right\rangle +\nu(t)\right]  .
\end{array}
\label{appen-rel}%
\end{equation}

\begin{proof}
The result can be obtained by applying It\^{o}'s formula.
\end{proof}

According to above Lemma, we have the following result which describes the
estimate of the solution $\left(  \hat{X}\left(  \cdot\right)  ,\hat{Y}\left(
\cdot\right)  ,\hat{Z}\left(  \cdot\right)  \right)  $. Before that, we need
impose the following assumption.

\begin{lemma}
\label{est-l} Suppose that the same assumptions in Lemma
\ref{appen-th-linear-fbsde} hold. Furthermore, suppose the FBSDE
(\ref{appen-eq-xyz}) has a unique solution in $L_{\mathcal{F}}^{\beta}%
(\Omega;C([0,T],\mathbb{R}^{n}))\times L_{\mathcal{F}}^{\beta}(\Omega
;C([0,T],\mathbb{R}))\times L_{\mathcal{F}}^{2,\beta}([0,T];\mathbb{R})$.
Then
\[%
\begin{array}
[c]{l}%
\mathbb{E}\left[  \sup\limits_{t\in\lbrack0,T]}\left(  |\hat{X}(t)|^{\beta
}+|\hat{Y}(t)|^{\beta}\right)  \right]  +\mathbb{E}\left[  \left(  \int%
_{0}^{T}|\hat{Z}\left(  t\right)  |^{2}dt\right)  ^{\frac{\beta}{2}}\right] \\
\leq C\mathbb{E}\left[  \left\vert x_{0}\right\vert ^{\beta}+\left\vert
\varsigma\right\vert ^{\beta}+\left(  \int_{0}^{T}\left(  \left\vert
L_{1}\left(  t\right)  |+|L_{3}(t)\right\vert \right)  dt\right)  ^{\beta
}+\left(  \int_{0}^{T}\left\vert L_{2}\left(  t\right)  \right\vert
^{2}dt\right)  ^{\frac{\beta}{2}}\right]  .
\end{array}
\]

\end{lemma}

\begin{proof}
The equation satisfied by $\left(  \varphi\left(  \cdot\right)  ,\nu\left(
\cdot\right)  \right)  $ is a linear BSDE with bounded coefficients. By
standard estimate of BSDE, we have%
\[%
\begin{array}
[c]{l}%
\mathbb{E}\left[  \sup\limits_{t\in\lbrack0,T]}|\varphi(t)|^{\beta}\right]
+\mathbb{E}\left[  \left(  \int_{0}^{T}|\nu\left(  t\right)  |^{2}dt\right)
^{\frac{\beta}{2}}\right] \\
\leq C\mathbb{E}\left[  \left\vert \varsigma\right\vert ^{\beta}+\left(
\int_{0}^{T}\left(  \left\vert L_{1}\left(  t\right)  |+|L_{3}(t)\right\vert
\right)  dt\right)  ^{\beta}+\left(  \int_{0}^{T}\left\vert L_{2}\left(
t\right)  \right\vert ^{2}dt\right)  ^{\frac{\beta}{2}}\right]  .
\end{array}
\]
From the result of above Lemma and the relation between $\left(  \hat
{Y}\left(  \cdot\right)  ,\hat{Z}\left(  \cdot\right)  \right)  $ and $\hat
{X}\left(  \cdot\right)  $, we obtain the estimate of $\hat{X}\left(
\cdot\right)  $ as follows
\[%
\begin{array}
[c]{l}%
\mathbb{E}\left[  \sup\limits_{t\in\lbrack0,T]}|\hat{X}(t)|^{\beta}\right] \\
\leq C\mathbb{E}\left[  \left\vert x_{0}\right\vert ^{\beta}+\left(  \int%
_{0}^{T}\left(  \left\vert L_{1}\left(  t\right)  |+|L_{2}\left(  t\right)
|+|\varphi\left(  t\right)  |+|\nu\left(  t\right)  \right\vert \right)
dt\right)  ^{\beta}+\left(  \int_{0}^{T}\left(  \left\vert L_{2}\left(
t\right)  |^{2}+|\varphi\left(  t\right)  |^{2}+|v\left(  t\right)
\right\vert ^{2}\right)  dt\right)  ^{\frac{\beta}{2}}\right] \\
\leq C\mathbb{E}\left[  \left\vert x_{0}\right\vert ^{\beta}+\left(  \int%
_{0}^{T}\left\vert L_{1}\left(  t\right)  \right\vert dt\right)  ^{\beta}%
+\sup\limits_{t\in\lbrack0,T]}|\varphi(t)|^{\beta}+\left(  \int_{0}^{T}\left(
\left\vert L_{2}\left(  t\right)  \right\vert ^{2}+\left\vert \nu\left(
t\right)  \right\vert ^{2}\right)  dt\right)  ^{\frac{\beta}{2}}\right] \\
\leq C\mathbb{E}\left[  \left\vert x_{0}\right\vert ^{\beta}+\left\vert
\varsigma\right\vert ^{\beta}+\left(  \int_{0}^{T}\left(  \left\vert
L_{1}\left(  t\right)  |+|L_{3}(t)\right\vert \right)  dt\right)  ^{\beta
}+\left(  \int_{0}^{T}\left\vert L_{2}\left(  t\right)  \right\vert
^{2}dt\right)  ^{\frac{\beta}{2}}\right]  .
\end{array}
\]
Since the relation (\ref{appen-rel}), we can obtain
\[
\mathbb{E}\left[  \sup\limits_{t\in\lbrack0,T]}|\hat{Y}(t)|^{\beta}\right]
\leq C\mathbb{E}\left[  \left\vert x_{0}\right\vert ^{\beta}+\left\vert
\varsigma\right\vert ^{\beta}+\left(  \int_{0}^{T}\left(  \left\vert
L_{1}\left(  t\right)  |+|L_{3}(t)\right\vert \right)  dt\right)  ^{\beta
}+\left(  \int_{0}^{T}\left\vert L_{2}\left(  t\right)  \right\vert
^{2}dt\right)  ^{\frac{\beta}{2}}\right]  ,
\]%
\[%
\begin{array}
[c]{rl}%
\mathbb{E}\left[  \left(  \int_{0}^{T}|\hat{Z}\left(  t\right)  |^{2}%
dt\right)  ^{\frac{\beta}{2}}\right]  & \leq C\mathbb{E}\left[  \left(
\int_{0}^{T}\left(  \left\vert \hat{X}\left(  t\right)  |^{2}+|L_{2}\left(
t\right)  |^{2}+|\varphi\left(  t\right)  |^{2}+|\nu\left(  t\right)
\right\vert ^{2}\right)  dt\right)  ^{\frac{\beta}{2}}\right] \\
& \leq C\mathbb{E}\left[  \sup\limits_{t\in\lbrack0,T]}\left(  \left\vert
\hat{X}\left(  t\right)  \right\vert ^{\beta}+|\varphi(t)|^{\beta}\right)
+\left(  \int_{0}^{T}\left(  \left\vert L_{2}\left(  t\right)  \right\vert
^{2}+\left\vert \nu\left(  t\right)  \right\vert ^{2}\right)  dt\right)
^{\frac{\beta}{2}}\right] \\
& \leq C\mathbb{E}\left[  \left\vert x_{0}\right\vert ^{\beta}+\left\vert
\varsigma\right\vert ^{\beta}+\left(  \int_{0}^{T}\left(  \left\vert
L_{1}\left(  t\right)  |+|L_{3}(t)\right\vert \right)  dt\right)  ^{\beta
}+\left(  \int_{0}^{T}\left\vert L_{2}\left(  t\right)  \right\vert
^{2}dt\right)  ^{\frac{\beta}{2}}\right]  .
\end{array}
\]
This completes the proof.
\end{proof}

\subsection{FBSDE with non-Lipschitz coefficients}

\begin{lemma}
\label{appen-th-linear-fbsde-unb}Suppose BSDE (\ref{appen-eq-pq}) has a unique
solution $(p(\cdot),q(\cdot))\in L_{\mathcal{F}}^{\infty}(\Omega
;C([0,T],\mathbb{R}^{n}))\times L_{\mathcal{F}}^{2,2}([0,T];\mathbb{R}^{n})$
such that $\left\vert 1-\left\langle p(t),\gamma_{2}(t)\right\rangle
\right\vert ^{-1}$ is bounded. Let
\begin{equation}
\left\{
\begin{array}
[c]{rl}%
d\tilde{X}(t)= & \left\{  \alpha_{1}(t)\tilde{X}(t)+\beta_{1}(t)\left\langle
p(t),\tilde{X}(t)\right\rangle +\gamma_{1}(t)\left\langle K_{1}\left(
t\right)  ,\tilde{X}(t)\right\rangle +L_{1}(t)\right. \\
& \left.  +\gamma_{1}(t)\left(  1-\left\langle p(t),\gamma_{2}(t)\right\rangle
\right)  ^{-1}\left\langle p(t),L_{2}(t)\right\rangle \right\}  dt\\
& +\left\{  \alpha_{2}(t)\tilde{X}(t)+\beta_{2}(t)\left\langle p(t),\tilde
{X}(t)\right\rangle +\gamma_{2}(t)\left\langle K_{1}\left(  t\right)
,\tilde{X}(t)\right\rangle +L_{2}(t)\right. \\
& \left.  +\gamma_{2}(t)\left(  1-\left\langle p(t),\gamma_{2}(t)\right\rangle
\right)  ^{-1}\left\langle p(t),L_{2}(t)\right\rangle \right\}  dB(t),\text{
}t\in\left[  0,T\right]  ,\\
\tilde{X}(0)= & x_{0}.
\end{array}
\right.  \label{appen-eq-x-unb}%
\end{equation}
Assume $\tilde{X}(\cdot)\in L_{\mathcal{F}}^{4}(\Omega;C([0,T],\mathbb{R}%
^{n}))$ and
\[
p(t)L_{1}(t)+q(t)L_{2}(t)+L_{3}(t)+(\gamma_{1}(t)p(t)+\gamma_{2}%
(t)q(t)+\gamma_{3}(t))(1-p(t)\gamma_{2}(t))^{-1}p(t)L_{2}(t)=0.
\]
Then $(\tilde{X}(\cdot),\tilde{Y}(\cdot),\tilde{Z}(\cdot))\in L_{\mathcal{F}%
}^{2}(\Omega;C([0,T],\mathbb{R}^{n}))\times L_{\mathcal{F}}^{2}(\Omega
;C([0,T],\mathbb{R}))\times L_{\mathcal{F}}^{2,2}([0,T];\mathbb{R})$ is the
unique solution to FBSDE (\ref{appen-eq-xyz}), where%
\begin{equation}%
\begin{array}
[c]{rl}%
\tilde{Y}(t)= & \left\langle p(t),\tilde{X}(t)\right\rangle ,\\
\tilde{Z}(t)= & \left\langle K_{1}\left(  t\right)  ,\tilde{X}(t)\right\rangle
+\left(  1-\left\langle p(t),\gamma_{2}(t)\right\rangle \right)
^{-1}\left\langle p(t),L_{2}(t)\right\rangle .
\end{array}
\label{appen-rel-unb}%
\end{equation}

\end{lemma}

\begin{proof}
Due to $p(\cdot)\in L_{\mathcal{F}}^{\infty}(\Omega;C([0,T],\mathbb{R}^{n}))$,
we have $\tilde{Y}(\cdot)\in L_{\mathcal{F}}^{2}(\Omega;C([0,T],\mathbb{R}))$.
On the other hand, from Theorem 5.2 in \cite{Hu-JX}, we can obtain
$q(\cdot))\in L_{\mathcal{F}}^{2,4}([0,T];\mathbb{R})$. Combining with
$\tilde{X}(t)\in L_{\mathcal{F}}^{4}(\Omega;C([0,T],\mathbb{R}^{n}))$ , we
have
\[%
\begin{array}
[c]{rl}%
\mathbb{E}\left[  \int_{0}^{T}|\left\langle K_{1}\left(  t\right)  ,\tilde
{X}(t)\right\rangle |^{2}dt\right]   & \leq C\mathbb{E}\left[  \sup
\limits_{t\in\lbrack0,T]}|\hat{X}(t)|^{2}\int_{0}^{T}\left(  1+|q(t)|^{2}%
\right)  dt\right]  \\
& \leq C\mathbb{E}\left[  \sup\limits_{t\in\lbrack0,T]}|\hat{X}(t)|^{2}%
\right]  +C\left\{  \mathbb{E}\left[  \sup\limits_{t\in\lbrack0,T]}|\hat
{X}(t)|^{4}\right]  \right\}  ^{\frac{1}{2}}\left\{  \mathbb{E}\left[  \left(
\int_{0}^{T}|q(t)|^{2}dt\right)  ^{2}\right]  \right\}  ^{\frac{1}{2}}\\
& <\infty\text{.}%
\end{array}
\]
This completes the proof. 
\end{proof}


\begin{thebibliography}{99}                                                                                               %




\bibitem {Dokuchaev-Zhou}M. Dokuchaev and X. Y. Zhou, Stochastic controls with
terminal contingent conditions. Journal of Mathematical Analysis and
Applications 238 (1999):pp. 143-165.

\bibitem {YingHu006}M. Fuhrman, Y. Hu and G. Tessitore, Stochastic maximum
principle for optimal control of SPDEs. Applied Mathematics \& Optimization,
68(2) (2013):pp. 181-217.



\bibitem {Hu17}M. Hu, Stochastic global maximum principle for optimization
with recursive utilities. Probability, Uncertainty and Quantitative Risk, 2(1)
(2017):pp 1-20.

\bibitem {Hu-JX}M. Hu, S. Ji and X. Xue, A global stochastic maximum principle
for fully coupled forward-backward stochastic systems. arXiv:1803.02109, (2018).

\bibitem {YingHu001}Y. Hu and S. Peng, Maximum principle for semilinear
stochastic evolution control systems. Stochastics and Stochastic Reports,
33(3-4) (1990): 159-180.

\bibitem {YingHu002}Y. Hu and S. Peng, Solution of forward-backward stochastic
differential equations. Probability Theory and Related Fields, 103(2)
(1995):pp. 273-283.

\bibitem {Ji-Zhou}S. Ji and X. Y. Zhou, A maximum principle for stochastic
optimal control with terminal state constrains and its applications.
Communications in Information \& Systems 6 (2006):pp. 321-337.



\bibitem {Ma-Yong-FBSDE}J. Ma and J. Yong, Forward-backward stochastic
differential equations and their applications. Springer Science \& Business
Media, (1999).

\bibitem {Ma-WZZ}J. Ma, Z. Wu, D. Zhang and J. Zhang, On well-posedness of
forward-backward SDEs-A unified approach. The Annals of Applied Probability,
25(4) (2015):pp. 2168-2214.

\bibitem {Ma-ZZ}J. Ma, J. Zhang and Z. Zheng, Weak Solutions for
Forward-Backward SDEs-A Martingale Problem Approach, Annals of Probability,
36(6) (2008):pp. 2092-2125.

\bibitem {Zhou003}T. Meyer-Brandis, B. {\O }ksendal and X. Y. Zhou, A
mean-field stochastic maximum principle via Malliavin calculus. Stochastics An
International Journal of Probability and Stochastic Processes, 84(5-6),
(2012):pp. 643-666.

\bibitem {Pardoux-book}E. Pardoux and A. Rascanu, Stochastic differential
equations, Backward SDEs, Partial differential equations. Springer, (2016).

\bibitem {Pardoux-Tang}E. Pardoux and S. Tang, Forward-backward stochastic
differential equations and quasilinear parabolic PDEs. Probability Theory and
Related Fields, 114(2), (1999):pp. 123-150.

\bibitem {Peng90}S. Peng, A general stochastic maximum principle for optimal
control problems. SIAM Journal on control and optimization, 28(4) (1990):pp. 966-979.

\bibitem {Peng93}S. Peng, Backward stochastic differential equations and
applications to optimal control. Applied mathematics \& optimization, 27(2)
(1993):pp. 125-144.

\bibitem {Peng99}S. Peng, Open problems on backward stochastic differential
equations Control of distributed parameter and stochastic systems. Springer
US, (1999):pp. 265-273.

\bibitem {Tang003}S. Tang, The maximum principle for partially observed
optimal control of stochastic differential equations. SIAM Journal on Control
and Optimization, 36(5) (1998):pp. 1596-1617.

\bibitem {Tang004}S. Tang and X. Li, Necessary conditions for optimal control
of stochastic systems with random jumps. SIAM Journal on Control and
Optimization, 32(5) (1994):pp. 1447-1475.

\bibitem {Shi-Wu}J. Shi and Z. Wu, The maximum principle for fully coupled
forward-backward stochastic control system. Acta Automatica. Sinica 32
(2006):pp. 161-169

\bibitem {Wu13}Z. Wu, A general maximum principle for optimal control of
forward-backward stochastic systems. Automatica, 49(5) (2013):pp. 1473-1480.

\bibitem {Xu95}W. Xu, Stochastic maximum principle for optimal control problem
of forward and backward system. The ANZIAM Journal 37 (1995):pp. 172-185.

\bibitem {YongZhou}J. Yong and X. Y. Zhou, Stochastic controls: Hamiltonian
systems and HJB equations. Springer, (1999).

\bibitem {Yong10}J. Yong, Optimality variational principle for controlled
forward-backward stochastic differential equations with mixed initial-terminal
conditions. SIAM Journal on Control and Optimization, 48(6) (2010):pp. 4119-4156.

\bibitem {Zhang17}J. Zhang. Backward Stochastic Differential Equations: From
Linear to Fully Nonlinear Theory (Vol. 86). Springer, (2017).


\end{thebibliography}
\end{document}